\documentclass[11pt]{amsart}
\usepackage{amscd, amsfonts, amsmath, amssymb, ulem, amsthm, mathrsfs, appendix, multicol, enumerate, graphicx, flexisym, epsfig, float, graphicx, hyperref, tikz, mathtools, comment, tikz-cd}
\usepackage[all,cmtip]{xy}
\usepackage{graphicx}
\usepackage{caption}
\usepackage{subcaption}
\makeatletter
\@namedef{subjclassname@2020}{Mathematics Subject Classification \textup{2020}}
\newsavebox{\@brx}
\newcommand{\llangle}[1][]{\savebox{\@brx}{\(\m@th{#1\langle}\)}%
  \mathopen{\copy\@brx\kern-0.5\wd\@brx\usebox{\@brx}}}
\newcommand{\rrangle}[1][]{\savebox{\@brx}{\(\m@th{#1\rangle}\)}%
  \mathclose{\copy\@brx\kern-0.5\wd\@brx\usebox{\@brx}}}
\newtheorem{theorem}{Theorem}[section]
\newtheorem{corollary}[theorem]{Corollary}
\newtheorem{lemma}[theorem]{Lemma}
\newtheorem{proposition}[theorem]{Proposition}
\theoremstyle{definition}
\newtheorem{conjecture}[theorem]{Conjecture}

\newtheorem{problem}[theorem]{Problem}
\newtheorem{definition}[theorem]{Definition}
\newtheorem{example}[theorem]{Example}
\newtheorem{remark}[theorem]{Remark}

\numberwithin{equation}{subsection}
\newtheorem*{ack}{Acknowledgement}

\usepackage[all,cmtip]{xy}

\newcommand{\Brun}{\mathrm{Brun}}

\newcommand{\id}{\mathrm{id}}

\begin{document}
\title{Brunnian planar braids and simplicial groups}

\author{Valeriy G. Bardakov} 
\address{Sobolev Institute of Mathematics, Acad. Koptyug ave. 4, 630090 Novosibirsk, Russia.}
\address{Novosibirsk State Agrarian University, Dobrolyubova str., 160, 630039 Novosibirsk, Russia.}
\address{Regional Scientific and Educational Mathematical Center of Tomsk State University, Lenin ave. 36, 634009 Tomsk, Russia.}
\email{bardakov@math.nsc.ru}

\author{Pravin Kumar}
\address{Department of Mathematical Sciences, Indian Institute of Science Education and Research (IISER) Mohali, Sector 81, SAS Nagar, P O Manauli, Punjab 140306, India} 
\email{pravin444enaj@gmail.com} 

\author{Mahender Singh}
\address{Department of Mathematical Sciences,Indian Institute of Science Education and Research (IISER) Mohali, Sector 81, SAS Nagar, P O Manauli, Punjab 140306, India} 
\email{mahender@iisermohali.ac.in} 


\begin{abstract}
Twin groups are planar analogues of Artin braid groups and play a crucial role in the Alexander-Markov correspondence for the isotopy classes of immersed circles on the 2-sphere without triple and higher intersections. These groups admit diagrammatic representations, leading to maps obtained by the addition and deletion of strands. This paper explores Brunnian twin groups, which are subgroups of twin groups composed of twins that become trivial when any of their strands are deleted. We establish that Brunnian twin groups consisting of more than two strands are free groups. Furthermore, we provide a necessary and sufficient condition for a Brunnian doodle on the 2-sphere to be the closure of a Brunnian twin. Additionally, we delve into two generalizations of Brunnian twins, namely, $k$-decomposable twins and Cohen twins, and prove some structural results about these groups. We also investigate a simplicial structure on pure twin groups that admits a simplicial homomorphism from Milnor's construction of the simplicial 2-sphere. This gives a possibility to provide a combinatorial description of homotopy groups of the 2-sphere in terms of pure twins.
\end{abstract}

\maketitle
\section{Introduction}
The twin group, or the planar braid group $T_n$ on $n \ge 2$ strands, is a right angled Coxeter group generated by $n-1$ involutions that admit only far commutativity relations. These groups appeared in the work of Khovanov \cite{MR1386845} on real $K(\pi, 1)$ subspace arrangements and were further investigated in \cite{MR1370644}. Twin groups have a geometrical interpretation  similar to the one for Artin braid groups \cite{MR1370644, MR1386845}. We fix parallel lines $y = 0$ and $y = 1$ on the plane $\mathbb{R}^2$ with $n$ marked points on each line. Consider the set of configurations of $n$ strands in the strip $\mathbb{R} \times [0,1]$ connecting the $n$ marked points on the line $y=1$ to those on the line $y=0$ such that each strand is monotonic and no three strands have a point in common. Two such configurations are equivalent if one can be deformed into the other by a homotopy of strands, keeping the end points fixed throughout the homotopy. Such an equivalence class is called a \textit{twin}. Placing one twin on top of another and rescaling the interval turns the set of all twins on $n$ strands into a group isomorphic to $T_n$. The generators $t_i$ of $T_n$ can be geometrically represented by configurations as shown in Figure~\ref{fig1}. 
 \begin{figure}[!ht]
 \begin{center}
\includegraphics[height=3cm]{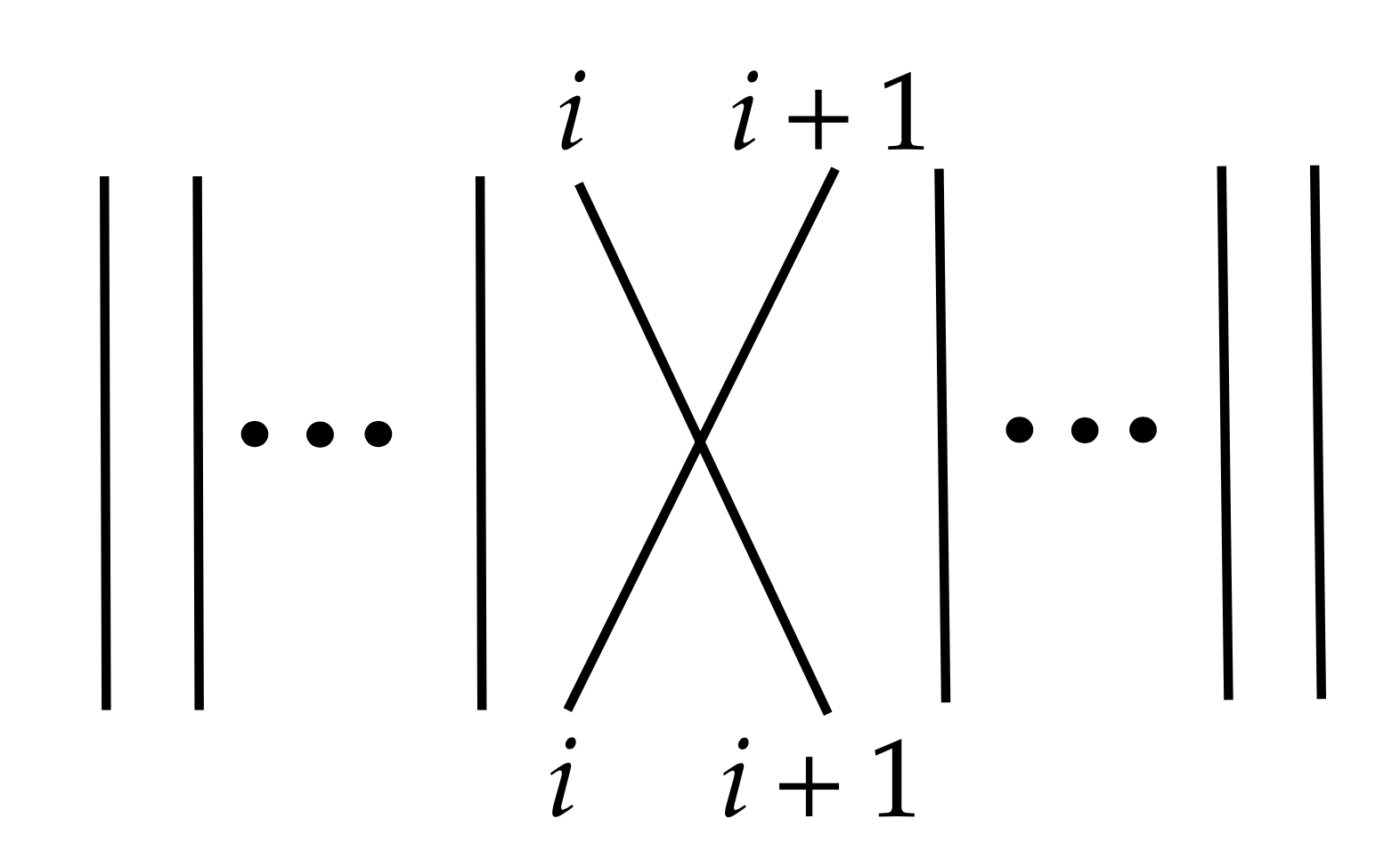}
\end{center}
\caption{The generator $t_i$ of $T_n$.} \label{fig1}
\end{figure}

Analogous to classical knot theory, it is evident that the closure of a twin gives a {\it doodle} on the 2-sphere. In general, a doodle on a closed surface is a collection of finitely many piecewise-linear closed curves without triple intersections. These objects first appeared in the work of Fenn and Taylor \cite{MR0547452}.  In \cite{MR1370644}, Khovanov proved that every oriented doodle on the 2-sphere is the closure of a twin.  A Markov theorem for doodles on the 2-sphere has been established by Gotin \cite{Gotin}, although the idea has been implicit in \cite{Khovanov1990}.  These constructions have been generalised by Bartholomew-Fenn-Kamada-Kamada \cite{MR3876348}, where they consider a collection of immersed circles in closed oriented surfaces of arbitrary genus.
\par 
The pure twin  group, denoted as $PT_n$, is defined as the kernel of the natural homomorphism from the twin group $T_n$ to the symmetric group $S_n$, which maps the twin $t_i$ to the transposition $(i, i + 1)$. A nice topological interpretation of $PT_n$ is  known due to Khovanov. Consider the space
$$
X_n = \mathbb{R}^n \setminus \big\{ (x_1, \ldots, x_n) \in \mathbb{R}^n  \, \mid \, x_i = x_j = x_k~\textrm{for}~ i \not= j \not= k \not= i \big\},
$$
which is the complement of triple diagonals $x_{i} = x_{j} = x_{k}$. In \cite{MR1386845}, Khovanov proved that the fundamental group $\pi_1(X_n)$ is isomorphic to $PT_n$. Prior to this, Bj\"{o}rner and Welker \cite{MR1317619} had investigated the cohomology of these spaces, establishing that each $H^i(X_n, \mathbb{Z})$ is free.
\par 

Simplicial structures on braid groups are connected with homotopy groups of some manifolds \cite{MR2188127,MR3152716,MR3180740, MR2551462, MR3035327}. Notably, they provide a description of elements in homotopy groups of the $2$-sphere in terms of Brunnian braids \cite{MR2188127}, with a generalization to higher dimensional spheres \cite{MR3035327}. A set of generators for the Brunnian braid group of a surface other than the 2-sphere and the projective plane has been provided in \cite{MR2966697}. Furthermore, Brunnian subgroups of mapping class groups have been considered in \cite{MR3108834}. In this paper, we explore  simplicial structures on pure twin groups. The geometrical interpretation of elements in twin groups allows us to define face and degeneracy maps obtained by the deletion and addition of strands, thereby transforming the family of pure twin groups into a simplicial group. We adopt the approach introduced by Cohen and Wu in \cite{MR2853222} for Artin pure braid groups.
\par

The paper is organised as follows. In Section \ref{section 2}, we prove that the natural maps of deletion and addition of strands turn the sequence $\{ T_n \}_{n \ge 1}$ into a  bi-$\Delta$-set, whereas the sequence $\{PT_n\}_{n \ge 1}$ is turned into a  bi-$\Delta$-group (Proposition \ref{bi-delta-set on tn}). In Section \ref{section 3}, we investigate Brunnian twins, which are twins that become trivial when any one of their strands is removed. We prove that the group $\Brun(T_n)$ of Brunnian twins on $n$ strands is free for $n\geq 3$ (Proposition \ref{brun tn free}), and give an infinite free generating set for $\Brun(T_4)$ (Theorem \ref{brun t4}). In Section \ref{section 4}, we consider two generalisations of Brunnian twins, namely,  $k$-decomposable twins and Cohen twins. A twin is $k$-decomposable if it becomes trivial after removing any $k$ of its strands. We give a complete description of $k$-decomposable twins on $n \ge 4$ strands (Proposition \ref{prop-D2}). A twin on $n$ strands is said to be Cohen if the twins obtained by removing any one of its strands are all the same. We give a characterisation for a twin to be Cohen  (Theorem \ref{P1pure}).  In Section \ref{section 5}, we consider Brunnian doodles on the 2-sphere, and prove that an $m$-component Brunnian doodle on the 2-sphere is the closure of a Brunnian twin if and only if its twin index is $ m$ (Theorem \ref{brunnian doodle twin index}). In Section \ref{section 6}, we observe that pure twin groups admit the structure of a simplicial group $SPT_*$. We relate it with the well-known Milnor's construction for simplicial spheres by establishing a homomorphism $\Theta : F[S^2]_* \longrightarrow  SPT_*$ of simplicial groups. We also identify some low degree terms of the image of $\Theta$ as free groups (Theorem \ref{k3 and k4}). A complete description of the image of $\Theta$ gives a possibility to provide a combinatorial description of homotopy groups of the 2-sphere in terms of pure twins.
\bigskip


\section{Bi-$\Delta$-set structure on twin and pure twin groups}\label{section 2}

For $n \ge 2 $, the {\it twin group} $T_n$ on $n$ strands is generated by $\{t_1, \ldots, t_{n-1} \}$ and it is defined by the following relations:
\begin{eqnarray*}
 t_{i}^2 & =& 1  \quad \mbox{for} \quad 1 \le i \le n-1
\end{eqnarray*}
and
\begin{eqnarray*}
  t_{i} t_{j} & =& t_{j} t_{i} \quad \mbox{for} \quad |i-j|\geq 2.
\end{eqnarray*}
Clearly, each $T_n$ is a right angled Coxeter group. Further, there is a surjective homomorphism  $\nu: T_n \to S_n$, that sends the generator $t_i$ to the transposition $\tau_i=(i,i+1)$ in the symmetric group $S_n$. It's kernel, denoted as $PT_n$, is called the pure twin group. It is not difficult to see that $PT_2$ is trivial and $PT_3$ is the infinite cyclic group generated by the pure twin $(t_1 t_2)^3$ \cite{MR4027588}. Figure \ref{pure twin in pt3} represents the pure twin  $(t_1 t_2)^3$.

 \begin{figure}[!ht]
 \begin{center}
\includegraphics[height=4.5cm]{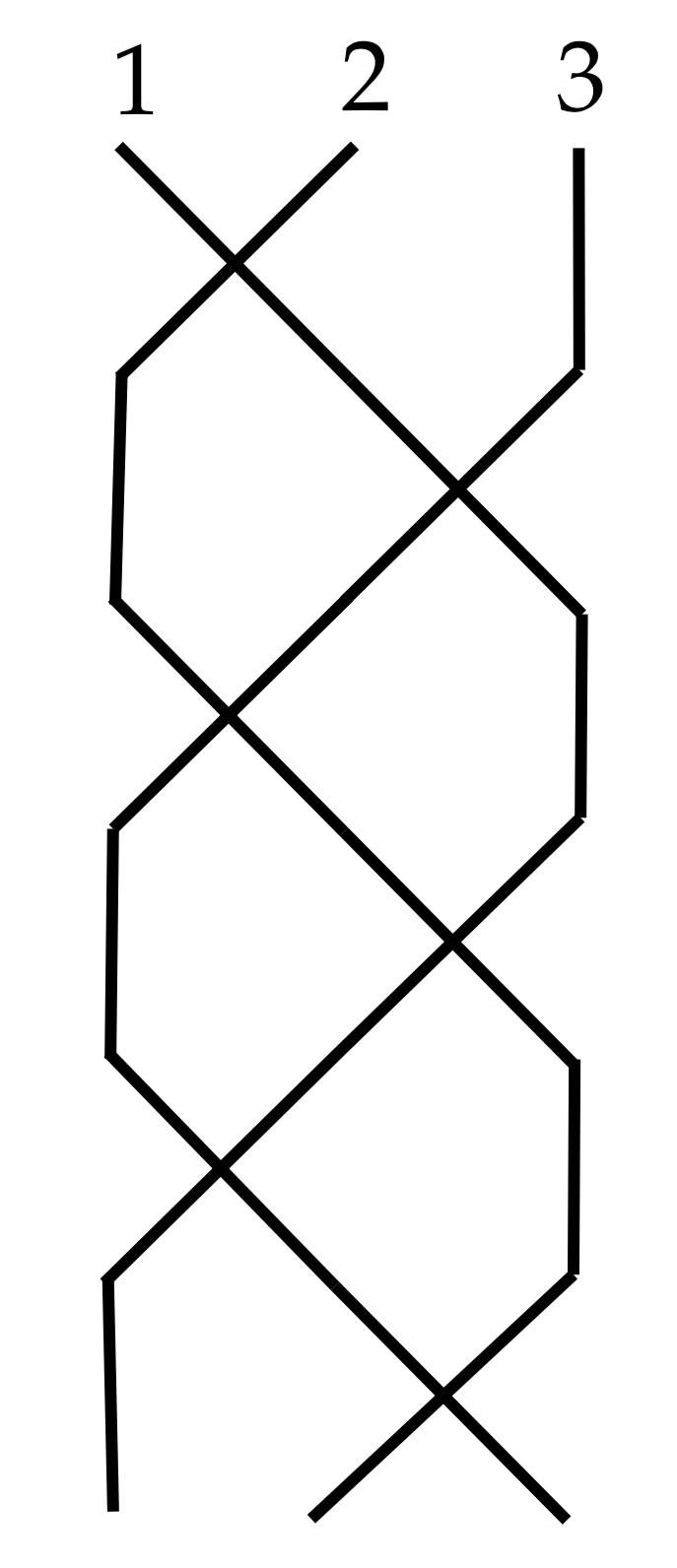}
\end{center}
\caption{The pure twin $(t_1 t_2)^3$.} 
\label{pure twin in pt3}
\end{figure}

\par

Let us consider the following definitions \cite[p.16]{MR2203532}.

\begin{definition}
A sequence of sets $\left\{G_n\right\}_{n \geq 0}$ is called a {\it $\Delta$-set} if there are maps $d_i: G_n \rightarrow G_{n-1}$ for each $0 \leq i \leq n$
such that
\begin{equation}\label{face map identity}
d_j d_i=d_i d_{j+1}
\end{equation}
 for all $j \geq i$. The maps $d_i$ are called face maps. If each $G_n$ is a group and each face map is a group homomorphism, then $\left\{G_n\right\}_{n \geq 0}$ is called a $\Delta$-group.
\end{definition}

\begin{definition}
A sequence of sets $\left\{G_n\right\}_{n \geq 0}$ is called a {\it bi-$\Delta$-set} if there are face maps $d_i: G_n \rightarrow G_{n-1}$ and coface maps $d^i: G_{n-1} \rightarrow G_n$ for each $0 \le i \le n$ such that the following identities  hold:
\begin{enumerate}
\item $d_j d_i=d_i d_{j+1}$ for  $j \geq i$,
\item $d^j d^i=d^{i+1} d^j$ for  $j \leq i$,
\item $ d_j d^i= d^{i-1} d_j$ for  $j<i$,
\item  $ d_j d^i=\id$ for $j=i$, 
\item  $ d_j d^i=  d^i d_{j-1}$ for $ j>i$.
\end{enumerate}
Moreover, if each $G_n$ is a group and each face and coface map is a group homomorphism, then $\left\{G_n\right\}_{n \geq 0}$ is called a bi-$\Delta$-group.
\end{definition}

We define a bi-$\Delta$-set structure on  twin groups that would induce a  bi-$\Delta$-group structure  on pure twin groups. For geometrical reasons, we take $G_n=T_{n+1}$ or $PT_{n+1}$ for each $n \ge 0$. For each $0 \leq i \leq n$, define the map $$d_i: T_{n+1}\to T_{n}$$ that deletes the $(i+1)$-th strand from the diagram of a twin on $n+1$ strands. Note that, $d_i$ is not a group homomorphism, but it  satisfies
\begin{equation}\label{di not homo on tn}
d_i(u w)= d_i(u)d_{\nu(u)(i+1)-1}(w)
\end{equation}
for all $u, w \in T_{n+1}$, where $\nu: T_{n+1} \to S_{n+1}$ is the natural surjection. On the other hand, we have $d_i(PT_{n+1})\subseteq PT_{n}$ for each  $0 \leq i \leq n$.  Further, it follows from \eqref{di not homo on tn} that $d_i: PT_{n+1} \to PT_{n}$ is a surjective group homomorphism  for each  $0 \leq i \leq n$. 

\begin{remark}
The  homomorphism $d_i: PT_{n+1} \to PT_{n}$ has an alternative interpretation. Consider the space
$$
X_{n}=\mathbb{R}^{n} \setminus \left\{\left(x_1, \ldots, x_n\right) \in \mathbb{R}^{n}~ \mid~ x_{i}=x_{j}=x_{k}, \quad i \neq j \neq k \neq i\right\},
$$
which is the complement of triple diagonals $x_{i}=x_{j}=x_{k}$ in $\mathbb{R}^{n} $. For each $1 \leq i \leq n+1$, let
$$
(p_i)_{\#}: \pi_1(X_{n+1}) \rightarrow \pi_1(X_n)
$$
be the group homomorphism induced by the coordinate projection $p_i:X_{n+1} \rightarrow X_n$, where 
$$p_i\left(x_1, \ldots, x_{n+1}\right)=\left(x_1, \ldots, x_{i-1}, x_{i+1}, \ldots, x_{n+1}\right).
$$
By \cite[Proposition 3.1]{MR1386845}, we identify $\pi_1(X_n)$ with $PT_n$, and observe that $(p_i)_{\#}=d_{i-1}$ for each $1 \leq i \leq n+1$.
\end{remark}
 
\par
In analogy with \cite[Example 1.2.8]{MR2203532}, we define a bi-$\Delta$-group structure on the family of twin groups.

\begin{proposition}\label{bi-delta-set on tn}
Consider the sequence of groups $\{ T_n \}_{n \ge 1}$. For each  $0 \leq i \leq n$, let $d_i: T_{n+1}\to T_{n}$ be the map satisfying \eqref{di not homo on tn} and $d^i: T_n \rightarrow T_{n+1}$ the map defined by
$$
d^i (t_j)=\left\{\begin{array}{ccc}
t_j & \text { for } \quad j<i, \\
t_{i+1} t_{i} t_{i+1} & \text { for }\quad j=i, \\
t_{j+1} & \text { for } \quad j>i.
\end{array}\right.
$$
Then  $\{ T_n, d_i, d^i \}_{n \ge 1}$ is a  bi-$\Delta$-set and $\{ PT_n, d_i, d^i \}_{n \ge 1}$ is a  bi-$\Delta$-group.
\end{proposition}

\begin{proof}
For each $0\leq i \leq n$, the map $d_i: T_{n+1} \rightarrow T_n$ clearly satisfies \eqref{face map identity}. Next, consider the maps
$$
d^i (t_j)=\left\{\begin{array}{ccc}
t_j & \text { for } \quad j<i, \\
t_{i+1} t_{i} t_{i+1} & \text { for }\quad j=i, \\
t_{j+1} & \text { for } \quad j>i.
\end{array}\right.
$$
See Figure \ref{Geometrical interpretation of a coface map} for a geometrical interpretation of these coface maps. A direct computation yields
$$
\begin{aligned}
d^j d^i (t_k) = d^{i+1} d^j(t_k)&=\left\{\begin{array}{ccc}
t_k & \text { for } \quad k<j\leq i, \\
t_{k+2} t_{k+1} t_{k}t_{k+1} t_{k+2} & \text { for }\quad j=k=i, \\
t_{k+1} t_{k} t_{k+1} & \text { for }\quad j=k < i, \\
t_{k+2} t_{k+1} t_{k+2} & \text { for }\quad j< k = i, \\
t_{k+1} & \text { for } \quad j < k < i, \\
t_{k+2} & \text { for } \quad j \leq i < k,
\end{array}\right.\\
\end{aligned}
$$
 for all $j \leq i$. This proves the identity (2). The identities (3)-(5) follows from the geometrical interpretation of $d_i$ and $d^i$. Hence, $\{ T_n, d_i, d^i \}_{n \ge 1}$ is a  bi-$\Delta$-set.
 \par
 
 We already noticed that, for each $0 \leq i \leq n$,  $d_i(PT_{n+1})\subseteq PT_{n}$ and  $d_i:PT_{n+1} \to PT_{n}$ is a group homomorphism. The inclusion $d^{i}(PT_{n})\subseteq PT_{n+1}$ follows from the geometrical interpretation of the map $d^i$. Alternatively, for each $n \ge 1$, let  $\eta^i:S_n\to S_{n+1}$ be the map defined by 
 $$
\eta^i (\tau_j)=\left\{\begin{array}{ccc}
\tau_j & \text { for } \quad j<i, \\
\tau_{i+1} \tau_{i} \tau_{i+1} & \text { for }\quad j=i, \\
\tau_{j+1} & \text { for } \quad j>i.
\end{array}\right.
$$
As with the case of $d^i$, each $\eta^i$ satisfies the far commutativity and involutory relations of generators of $S_n$. For braid relations, we see that

$$
\begin{aligned}
\eta^i(\tau_k)\eta^i(\tau_{k+1})\eta^i(\tau_k)&=\left\{\begin{array}{ccc}
\tau_k \tau_{k+1} \tau_{k} & \text { for } \quad k+1 < i, \\
 \tau_{k}\tau_{k+1}\tau_{k+2}\tau_{k+1}\tau_{k} & \text { for } i = k, k+1, \\
\tau_{k+1}\tau_{k+2}\tau_{k+1} & \text { for }\quad k > i,
\end{array}\right.\\
&=\eta^i(\tau_{k+1})\eta^i(\tau_{k})\eta^i(\tau_{k+1}),
\end{aligned}
$$
and hence each $\eta^i$ is a group homomorphism. Then the inclusion $d^{i}(PT_{n})\subseteq PT_{n+1}$ also follows from the commutativity of the following diagram
$$
\begin{tikzcd}
	{T_n} & {T_{n+1}} \\
	{S_n} & {S_{n+1}}
	\arrow["{d^{i}}", from=1-1, to=1-2]
	\arrow["{\eta^i}", from=2-1, to=2-2]
	\arrow["\nu"', from=1-1, to=2-1]
	\arrow["\nu", from=1-2, to=2-2].
\end{tikzcd}
$$
Finally, we prove that each $d^i$ is group homomorphism at the level of twin groups itself. Clearly,  $(d^i(t_k))^2=1$ for all $i$ and $k$. Further, for $k<\ell$ with $|k-\ell|\geq 2$, we have
$$
\begin{aligned}
d^i(t_k)d^i(t_\ell)&=\left\{\begin{array}{ccc}
t_kt_\ell & \text {for} \quad \ell < i, \\
t_kt_{\ell+1}t_\ell t_{\ell+1} & \text {for} \quad \ell=i, \\
t_{k} t_{\ell+1} & \text {for}\quad  k < i<\ell, \\
t_{k+1} t_{k} t_{k+1}t_{\ell+1} & \text {for}\quad k = i, \\
t_{k+1}t_{\ell+1} & \text {for} \quad i < k,
\end{array}\right.\\
&=d^i(t_\ell)d^i(t_k).
\end{aligned}
$$
This proves that $\{ PT_n, d_i, d^i \}_{n \ge 1}$ is a  bi-$\Delta$-group.
 \end{proof}

 \begin{figure}[H]
\includegraphics[width=4.5cm]{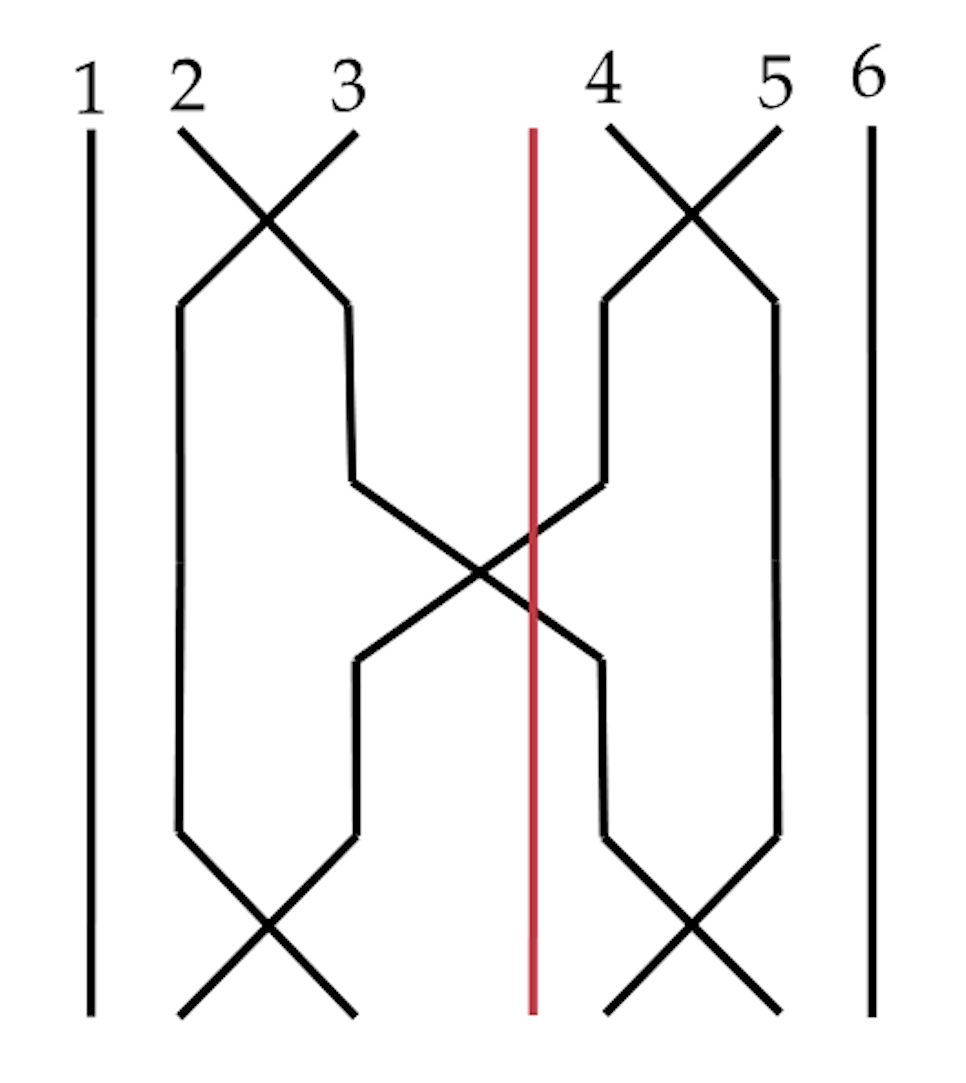}
\caption{The geometrical interpretation of a coface map.}
\label{Geometrical interpretation of a coface map}
\end{figure}

 \begin{remark}\label{another definition of di}
For each $0\leq i \leq n$, we can also define the coface maps  $d^i: T_n \rightarrow T_{n+1}$ by
$$
d^i (t_j)=\left\{\begin{array}{ccc}
t_j & \text { for } \quad j<i, \\
t_{i} t_{i+1} t_{i} & \text { for }\quad j=i, \\
t_{j+1} & \text { for } \quad j>i.
\end{array}\right.
$$
It can be verified that the analogue of Proposition \ref{bi-delta-set on tn} holds with these coface maps.
\end{remark} 
\bigskip

We now use the bi-$\Delta$-set structure on $\{T_n\}_{n \ge 1}$ to give a new presentation for $T_{n+1}$. We use the coface maps $d^i$ as defined in Proposition \ref{bi-delta-set on tn}.
\begin{proposition}
Let $q_k:=d^{n}d^{n-1}\cdots d^{k}(t_{k})$ for $1 \le k\le n-1$ and $q_n:=d^{n-2}(t_{n-1})$. Then $T_{n+1}$ admits a presentation with generating set $\{q_1, \ldots, q_n\}$ and the following defining relations:
\begin{enumerate}
\item $q_i^2=1$ for all $i$,
\item $[ q_{i+1} q_i q_{i+1}, q_n] = 1$ for $i < n-1$,
\item $[q_{i+1}q_iq_{i+1},q_{j+1}q_jq_{j+1}]$ for $|i-j|>2$ and $i,j \leq n-1$.
\end{enumerate}
\end{proposition}
\begin{proof}
Using the simplicial identity $d^j d^i = d^{i+1} d^j$ for  $j \le i$, we can assume that $i_{n-1}>i_{n-2}>\cdots>i_1$ in the composite map $d^{i_{n-1}}d^{i_{n-2}} \cdots d^{i_{1}}$. We see that 
$$q_k=d^{n-1}d^{n-2}\cdots d^{k}(t_k )=t_nt_{n-1}\cdots t_{k+1}t_kt_{k+1}\cdots t_{n-1} t_n~\text {for}~1 \le k\le n-1$$
and
$$q_n=d^{n-2}(t_{n-1})=t_n.$$
A direct check gives $t_k=q_{k+1}q_k q_{k+1}$ for each $1 \le k \le n-1$, and hence $\{q_1, \ldots, q_n\}$ generates $T_{n+1}$. Further, the defining relations of $T_{n+1}$ in terms of the Coxeter generating set $\{t_1, \ldots, t_n \}$ gives the defining relations for the new generating set as follows:
\begin{enumerate}
\item $q_i^2=1$ for all $i$,
\item $[ q_{i+1} q_i q_{i+1}, q_n] = 1$ for $i < n-1$,
\item $[q_{i+1}q_iq_{i+1},q_{j+1}q_jq_{j+1}]$ for $|i-j|>2$ and $i,j \leq n-1$.
\end{enumerate}
\end{proof}
\medskip

\section{Brunnian twins}\label{section 3}
In the influential work \cite{MR2188127}, a connection has been established between certain quotients of the Brunnian braid groups of the 2-sphere and its higher homotopy groups.

\begin{definition}
A pure twin is said to be \textit{Brunnian} if it becomes trivial after removing any one of its strands. 
\end{definition}

Let $\Brun(T_n)$ denote the set of all Brunnian twins on $n$ strands.

\begin{proposition}\label{BrNormalInTn}
$\Brun(T_n)$ is a normal subgroup of $PT_n$.
\end{proposition}
\begin{proof}
For each $0 \leq i \leq n-1$, let $d_i: PT_n\to PT_{n-1}$ be the face map of Proposition \ref{bi-delta-set on tn}. Since each $d_i$ is a group homomorphism and $$\Brun(T_n) = \bigcap_{i=0}^{n-1} \ker (d_i),$$
it follows that $\Brun(T_n)$ is a normal subgroup of $T_n$.
\end{proof}

Next, we attempt to understand the groups of Brunnian twins.

\begin{proposition}
For $n \geq 4$, $\Brun(T_n)$ does not contain any element from $\{(t_it_{i+1})^3 \mid 1 \le i \le n-1 \}$.
\end{proposition}
\begin{proof}
For $n \ge 4$,  removing a trivial strand from $(t_it_{i+1})^3$ gives a non-trivial twin, and hence the assertion follows.
\end{proof}

\begin{proposition}\label{brun t3 is pt3}
 $\Brun(T_3) \cong PT_3 \cong \mathbb{Z}$.
\end{proposition}
\begin{proof}
We already have $\Brun(T_3) \subseteq PT_3$. By \cite[Theorem 2]{MR4027588}, we know that $PT_3$ is the infinite cyclic group generated by $(t_1t_2)^3$ (see Figure \ref{pure twin in pt3}), and clearly $(t_1t_2)^3 \in \Brun(T_3)$.
\end{proof}

In contrast, it is proved in \cite{MR362287} that $\Brun(B_3)$ is the commutator subgroup of the Artin pure braid group $P_3$. 

\begin{theorem}\label{brun t4}
$\Brun(T_4)$ is a free group of infinite rank.
\end{theorem}

\begin{proof}
By \cite[Theorem 2]{MR4027588}, $PT_4$ is a free group of rank $7$ generated by 
$$x_1=(t_1t_2)^3, \quad x_2=((t_1t_2)^3)^{t_3}, \quad  x_3=((t_1t_2)^3)^{t_3t_2}, \quad  x_4=((t_1t_2)^3)^{t_3t_2t_1},$$
$$ x_5=(t_2t_3)^3, \quad  x_6=((t_2t_3)^3)^{t_1}, \quad  x_7=((t_2t_3)^3)^{t_1t_2}.$$ 
Denote the generator $(t_1t_2)^3$  of $PT_3$ by $y$. Direct computations show that images of $x_i$'s under  the face maps $d_i$'s are as follows:
\begin{eqnarray*}
d_0(x_1)=d_0(x_2)=d_0(x_3)= d_0(x_6)=d_0(x_7)=1, && d_0(x_4)=d_0(x_5)=y,\\
d_1(x_1)=d_1(x_2)=d_1(x_4)= d_1(x_5)=d_1(x_7)=1, && d_1(x_3)=d_1(x_6)=y,\\ 
d_2(x_1)=d_2(x_3)=d_2(x_4)= d_2(x_5)=d_2(x_6)=1, && d_2(x_2)=d_2(x_7)=y,\\
d_3(x_2)=d_3(x_3)=d_3(x_4)= d_3(x_5)=d_3(x_6)=d_3(x_7)=1, && d_3(x_1)=y.
\end{eqnarray*}
For each generator $x_i$, let $\log_i(w)$ denote the sum of the powers of $x_i$ in the word $w$. Then it follows that
\begin{eqnarray*}
\ker(d_0) &=& \lbrace w \in PT_4~\mid~ \log_{4}(w)+ \log_{5}(w) = 0 \rbrace,\\
\ker(d_1) &=& \lbrace w \in PT_4~\mid~ \log_{3}(w)+ \log_{6}(w) = 0 \rbrace,\\
\ker(d_2) &=& \lbrace w \in PT_4~\mid~ \log_{2}(w)+ \log_{7}(w) = 0 \rbrace,\\
\ker(d_3) &=& \lbrace w \in PT_4~\mid~ \log_{1}(w) = 0 \rbrace,
\end{eqnarray*}
and hence
\begin{eqnarray*}
&& \Brun(T_4)\\
&=& \bigcap_{i=0}^3\ker(d_i) \\
&=& \Big \lbrace w \in PT_4~\mid~  \log_{4}(w)+ \log_{5}(w) =  \log_{3}(w)+ \log_{6}(w) = \log_{2}(w)+ \log_{7}(w) =\log_{1}(w) = 0\Big \rbrace.
\end{eqnarray*}
Clearly, $\Brun(T_4)$ is  free being a subgroup of the free group $PT_4$. We now find an infinite free basis for $\Brun(T_4)$. It follows from the preceding description of $\Brun(T_4)$ that the commutator subgroup of $PT_4$ is contained in $\Brun(T_4)$. In fact, the containment is strict since $x_4x_5^{-1} \in \Brun(T_4)$, but  $x_4x_5^{-1} \not\in PT_n^\prime$. Thus, $PT_4/\Brun(T_4)$ is a non-trivial abelian group. Let $q:PT_4 \to PT_4/\Brun(T_4)$
be the quotient map with $q(x_i)=y_i$ for  $1 \le i \le 7$. Since $x_2x_7^{-1} , x_3x_6^{-1}, x_4x_5^{-1} \in \Brun(T_4)$, the  group $PT_4/\Brun(T_4)$ is generated by the set $\{y_1, y_2, y_3, y_4\}$. Note that $x_ix_j^{-1} \not\in \Brun(T_4)$ for all $i \neq j \in \{1,2,3,4\}$ and $x_i^k \not\in \Brun(T_4)$ for $k>0$. Thus, by the fundamental theorem for finitely generated abelian groups, $PT_4/\Brun(T_4)$ is a free abelian group of rank 4.
\par
Consider the short exact sequence
$$
1\to \Brun(T_4) \to PT_4 \to \mathbb Z^4 \to 1.
$$
We fix a Schreier system $\{x_1^{k_1} x_2^{k_2} x_3^{k_3} x_4^{k_4} \mid k_1, k_2, k_3, k_4 \in \mathbb Z\}$ of coset representatives of $\Brun(T_4)$ in $PT_4$. This gives a free basis for $\Brun(T_4)$ consisting of elements of the  form
\begin{eqnarray*}
x_1^{k_1} x_2^{k_2} x_3^{k_3} x_4^{k_4} x_1 (x_1^{k_1+1} x_2^{k_2} x_3^{k_3} x_4^{k_4})^{-1}, && x_1^{k_1} x_2^{k_2} x_3^{k_3} x_4^{k_4} x_2 (x_1^{k_1} x_2^{k_2+1} x_3^{k_3} x_4^{k_4})^{-1},\\
x_1^{k_1} x_2^{k_2} x_3^{k_3} x_4^{k_4} x_3 (x_1^{k_1} x_2^{k_2} x_3^{k_3+1} x_4^{k_4})^{-1}, && x_1^{k_1} x_2^{k_2} x_3^{k_3} x_4^{k_4} x_4 (x_1^{k_1} x_2^{k_2} x_3^{k_3} x_4^{k_4+1})^{-1},\\
x_1^{k_1} x_2^{k_2} x_3^{k_3} x_4^{k_4} x_5 (x_1^{k_1} x_2^{k_2} x_3^{k_3} x_4^{k_4+1})^{-1}, && x_1^{k_1} x_2^{k_2} x_3^{k_3} x_4^{k_4} x_6 (x_1^{k_1} x_2^{k_2} x_3^{k_3+1} x_4^{k_4})^{-1},\\
x_1^{k_1} x_2^{k_2} x_3^{k_3} x_4^{k_4} x_7 (x_1^{k_1} x_2^{k_2+1} x_3^{k_3} x_4^{k_4})^{-1}, &&
\end{eqnarray*}
for $k_1, k_2, k_3, k_4 \in \mathbb Z$. This completes the proof.
\end{proof}

\begin{remark}
$PT_n/\Brun(T_n)$ is non-abelian for $n\geq 5$ since $PT_n^{\prime} \not\subseteq \Brun(T_n)$ for $n\geq 5$. For example,  if $w=[(t_1t_2)^3, (t_2t_3)^3]$, then $d_4(w)\neq 1$.
\end{remark}

\begin{proposition}
$PT_n/\Brun(T_n)$ is torsion free for each $n \ge 4$ and $PT_4/\Brun(T_4) \cong \mathbb{Z}^4$.
\end{proposition}
\begin{proof}
The homomorphisms $d_i: PT_n \rightarrow PT_{n-1}$ induce an injective homomorphism
$$
PT_n / \Brun(T_n) \hookrightarrow  \underbrace{PT_{n-1} \times \cdots \times PT_{n-1}}_{n~times}.
$$
By \cite[Theorem 3]{MR4027588}, $PT_n$ is torsion-free for $n \ge 3$. Hence, $PT_{n-1} \times \cdots \times PT_{n-1}$ is torsion free, and therefore  $PT_n / \Brun(T_n)$ is so. The second assertion is proven in the proof of Theorem \ref{brun t4}. 
\end{proof}
\par

\begin{problem}
Describe the structure of the group $PT_n/\Brun(T_n)$  for $n \ge 5$.
\end{problem}
\medskip

Recall from \cite{MR4027588, MR4651964} that  the {\it virtual twin group} $VT_n$ on $n \ge 2$ strands is the group generated by $\{ t_1, \ldots, t_{n-1}, \rho_1, \ldots, \rho_{n-1}\}$ and having the following defining relations:
\begin{eqnarray*}
t_i^{2} &=&1 \quad \textrm{for} \quad1 \le i \le  n-1, \label{1}\\ 
t_i t_j &=& t_j t_i\quad \textrm{for} \quad |i - j| \geq 2,\label{2}\\
\rho_i^{2} &=& 1\quad \textrm{for} \quad 1 \le i \le n-1, \label{3}\\
\rho_i\rho_j &=& \rho_j\rho_i \quad \textrm{for} \quad |i - j| \geq 2, \label{4}\\
\rho_i\rho_{i+1}\rho_i &=& \rho_{i+1}\rho_i\rho_{i+1} \quad\textrm{for} \quad 1 \le i \le n-2, \label{5}\\
\rho_i t_j &=& t_j\rho_i \quad \textrm{for} \quad |i - j| \geq 2, \label{6}\\
\rho_i\rho_{i+1} t_i &=& t_{i+1} \rho_i \rho_{i+1} \quad \textrm{for} \quad 1 \le i \le n-2. \label{7}
\end{eqnarray*}

The group $VT_n$ plays the role of virtual braid groups in the Alexander-Markov correspondence for the planar analogue of virtual knot theory. There is a surjective homomorphism $\mu:VT_n  \to S_n$ given by $$\mu(t_i) = \mu(\rho_i) = (i, i+1)$$
for all $1\leq i \leq n-1$. The kernel $PVT_n$ of this surjection is called the \textit{pure virtual twin group} on $n$ strands.
\par 

For each $n\geq 2$, we  have surjective homomorphisms $d_{n-1}: PT_n \to PT_{n-1}$ and $\overline{d}_{n-1}: PVT_n \to PVT_{n-1}$ that delete the $n$-th strand from the diagram of a pure twin and pure virtual twin. In the reverse directions, we have homomorphisms $d^{n-1}: PT_{n-1}\to PT_n$ and $\overline{d}^{n-1}: PVT_{n-1}\to PVT_n$ that add a trivial strand to the right side of the diagram. Further, we have $d_{n-1} \, d^{n-1}= \id_{PT_{n-1}}$ and $\overline{d}_{n-1} \, \overline{d}^{n-1}= \id_{PVT_{n-1}}$. Setting $U_n=\ker(d_{n-1})$ and $V_n=\ker(\overline{d}_{n-1})$, we have split short exact sequences
$$
1 \to U_n \to PT_n \to PT_{n-1} \to 1
$$
and
$$
1 \to V_n \to PVT_n \to PVT_{n-1} \to 1.
$$

In other words, $PT_n \cong U_n \rtimes PT_{n-1}$ and $PVT_n \cong V_n \rtimes PVT_{n-1}$.
\par

\begin{proposition}\label{brun tn free}
$\Brun(T_n)$ is free for all $n\geq 3$.
\end{proposition}

\begin{proof}
The map $t_i\mapsto t_i$ gives an embedding of $T_n$ into $PVT_n$ \cite[Corollary 3.5]{NNS2}. Restricting to $PT_n$, this gives an inclusion $\psi_n: PT_n \to PVT_n$ such that the following diagram commutes
$$
\begin{tikzcd}
PT_n \arrow[r, "d_{n-1}"] \arrow[d, "\psi_n"]
& PT_{n-1} \arrow[d, "\psi_{n-1}"] \\
PVT_n \arrow[r, "\overline{d}_{n-1}"] &  PVT_{n-1}.
\end{tikzcd}
$$

This gives  $U_n \cong \psi_n(U_n) = \psi_n (\ker(d_{n-1})) \leq   \ker(\overline{d}_{n-1}) = V_n$. Since $V_n$ is free for $n \ge 2$ \cite[Theorem 4.1]{MR4651964}, it follows that $U_n$ is also free. Note that the subgroup $U_i=\ker(d_i)$ is conjugate to $U_n$ by the element $t_{n-1}t_{n-2}\cdots t_{i+1}$. Thus, $U_i$ is free group for each $1 \le i \le n$, and hence $\Brun(T_n)=\cap_{i=1}^{n} U_i$ is a free group. 
\end{proof}

At this juncture, the ensuing problem naturally arises.

\begin{problem}
Determine a free generating set for $\Brun(T_n)$ for $n \ge 5$.
\end{problem}

We conclude the section with a consequence of the Decomposition Theorem for bi-$\Delta$-groups in our setting \cite[Proposition 1.2.9]{MR2203532}.

\begin{proposition}\label{prop:ptndecomposition}
The pure twin group $PT_{n+1}$ is the iterated semi-direct product of subgroups 
$$
\big\{ d^{i_k} d^{i_{k-1}} \cdots d^{i_1}(\Brun(T_{n-k+1})) ~\mid~ 0 \leq i_1<i_2<\cdots<i_k \leq n~\textrm{and}~0 \leq k \leq n \big\}
$$
with the lexicographic order on the indexing set $$ \big\{ ({i_k}, {i_{k-1}}, \ldots, {i_1},\underbrace{{i_0},{i_0},\ldots,{i_0}}_{n-k \text{ times }}) \mid 0 \leq i_1<i_2<\cdots<i_k \leq n~\textrm{and}~0 \leq k \leq n \big\}$$ from the left, where ${i_0}$ is the blank symbol considered smaller than all other indices.
\end{proposition}
\begin{example}
For $n=3$, we have
\begin{eqnarray*}
d^0(PT_3) = \langle (t_2t_3)^3 \rangle, && d^1(PT_3) = \langle (t_2t_1t_2t_3)^3 \rangle,\\
d^2(PT_3) = \langle (t_1t_3t_2t_3)^3 \rangle , && d^3(PT_3) = \langle (t_1t_2)^3 \rangle.
\end{eqnarray*}
\begin{figure}[H]
\includegraphics[width=13cm]{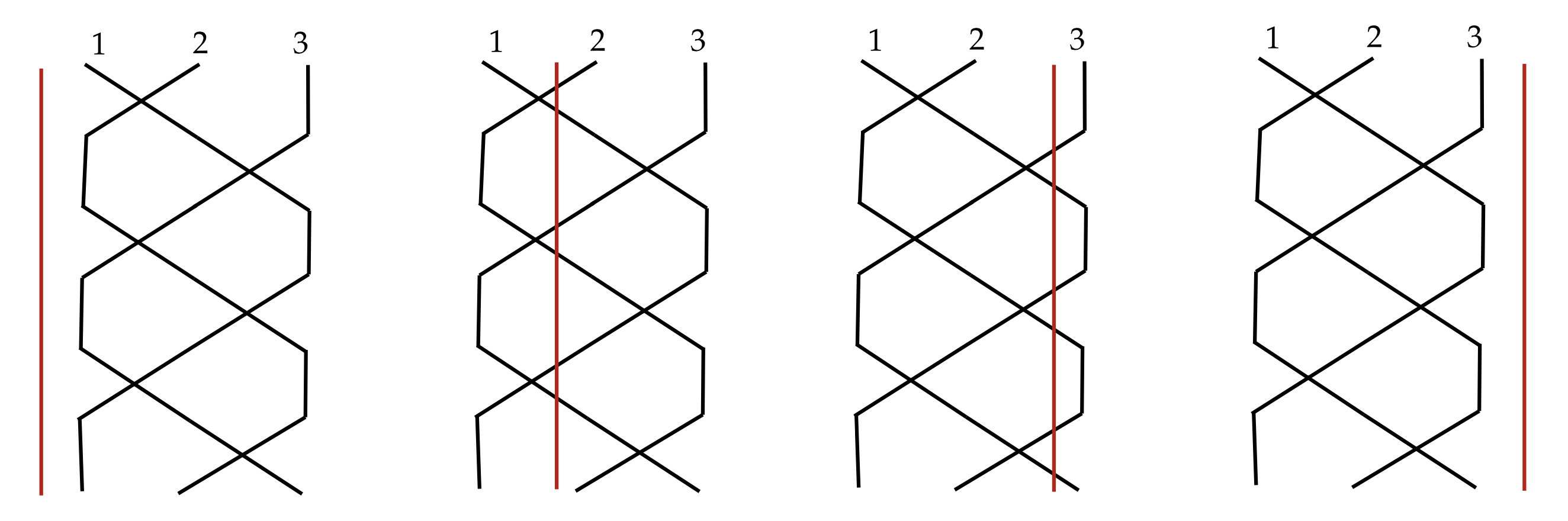}
\caption{Images of $(t_1t_2)^3$ under the coface maps $d^0, d^1, d^2$ and $d^3$.} 
\end{figure}
Observing the proof of \cite[Proposition 1.2.9]{MR2203532}, we get
$$PT_4 = \ker(d_3)\rtimes \langle (t_1t_2)^3 \rangle,$$
where $\ker(d_3)$ is normal in $PT_4$ and $\langle (t_1t_2)^3$ acts on $\ker(d_3)$ via conjugation. At the second stage, we obtain $$\ker(d_3) = (\ker(d_2)\cap \ker (d_3)) \rtimes \langle (t_1t_3t_2t_3)^3 \rangle,$$
where $\ker(d_2)\cap \ker (d_3)$ is normal in $\ker(d_3)$ and the subgroup $\langle (t_1t_3t_2t_3)^3 \rangle = \langle d^2((t_1t_2)^3) \rangle \leq \ker(d_3)$ acts on $\ker(d_2)\cap \ker (d_3)$ via conjugation. At the third stage, we get
$$
\ker(d_2)\cap \ker (d_3) = (\ker(d_1) \cap \ker(d_2)\cap \ker (d_3)) \rtimes \langle (t_2t_1t_2t_3)^3 \rangle,
$$
where $\ker(d_1)\cap \ker(d_2)\cap \ker (d_3)$ is normal in $\ker(d_2)\cap \ker(d_3)$ and the subgroup $\langle (t_2t_1t_2t_3)^3 \rangle = \langle d^1((t_1t_2)^3) \rangle \leq \ker(d_2)\cap\ker(d_3)$ acts on $\ker(d_1)\cap \ker(d_2)\cap \ker (d_3)$ via conjugation. Finally, we have
$$
\ker(d_1)\cap\ker(d_2)\cap \ker (d_3) = \Brun(T_4) \rtimes \langle (t_2t_3)^3\rangle,
$$
where $\Brun(T_4)$ is normal and the subgroup $\langle (t_2t_3)^3\rangle= \langle d^0((t_1t_2)^3) \rangle$ acts on $\Brun(T_4)$ via conjugation. Thus, we obtain the following decomposition of $PT_4$ as  an iterated semi-direct product
$$
PT_4 =\left(\left(\left(\Brun(T_4) \rtimes \langle (t_2t_3)^3\rangle\right) \rtimes \langle (t_2t_1t_2t_3)^3 \rangle\right) \rtimes \langle (t_1t_3t_2t_3)^3 \rangle\right) \rtimes \langle (t_1t_2)^3 \rangle.
$$
Similarly, there are 16 non-trivial terms in the decomposition of $PT_5$ with the leftmost term being the Brunnian subgroup $\Brun(T_5)$.
\end{example}
\medskip


\section{$k$-decomposable twins and Cohen twins}\label{section 4}
In this section, we consider two generalisations of Brunnian twins.

\subsection{$k$-decomposable twins} We begin with the following definition.

\begin{definition}
A pure twin on $n$ strands is said to be \textit{$k$-decomposable} if it becomes trivial after removing any $k$ of its strands.
\end{definition}

Clearly, a 1-decomposable twin is simply a Brunnian twin. Further, the set of all $k$-decomposable twins on $n$ strands forms a normal subgroup of $PT_n$ and we denote this subgroup by $D_{k,n}$. For $w\in PT_n$ and $1\leq i<j< k\leq n$, let $w_{i,j,k}$ be the pure twin obtained from $w$ by deleting all the strands except those indexed $i,j,k$. We can still view each $w_{i,j,k}$ as an element of $PT_n$ by adding trivial $(n-3)$ strands on its right. See Figure \ref{Decomposing a pure twin into Brunnian twins} for an example for $n=4$. Using ideas from \cite{MR324684}, we prove the following result.

\begin{proposition}\label{prop-D1}
For $n \ge 4$, $$D_{n-3,n}= \big\{ w \prod_{1\leq i<j< k\leq n} (w_{i, j, k}^{-1})^{c_{i,j,k}}~\mid ~w\in PT_n \big\},$$
where $c_{i,j,k} \in T_n$ is a coset representative of the permutation in $T_n/PT_n\cong S_n$ which takes $i , j, k$ to $1,2,3$, respectively, and fix everything else.
\end{proposition}

\begin{proof}
In view of Proposition \ref{brun t3 is pt3}, we have $w_{i,j,k}\in \Brun(T_3)$. A direct check shows that for any $w \in PT_n$, the pure twin $$w \prod_{1\leq i<j< k\leq n} (w_{i, j, k}^{-1})^{c_{i,j,k}}$$ is a $(n-3)$-decomposable twin on $n$ strands. Note that the map $\phi:PT_n \to  D_{n-3,n}$ given by $$\phi(w)=  w \prod_{1\leq i<j< k\leq n} (w_{i, j, k}^{-1})^{c_{i,j,k}}$$ is a retraction, that is, the restriction of $\phi$ on $D_{n-3,n}$ is the identity map. Hence, it follows that each element of  $D_{n-3,n}$ arises in this fashion.
\end{proof}

\begin{corollary}\label{prop-BT4}
$\Brun (T_4)= \big\{ww_{1,2,3}^{-1} (w_{1,2,4}^{-1})^{t_3}(w_{1,3,4}^{-1})^{t_2t_3}(w_{2,3,4}^{-1})^{t_1t_2t_3} ~\mid ~ w\in PT_4 \big\}$.
\end{corollary}

Next, we describe a process of constructing $D_{k-1,n}$ from $D_{k,n}$. Let $w \in D_{k,n}$ and $1\leq i_1<i_2\cdots< i_{n-k+1}\leq n$. Let $w_{i_1,i_2,\ldots,i_{n-k+1}}$ be the pure twin obtained from $w$ by removing the $k-1$ strands except those indexed $i_1,i_2,\ldots,i_{n-k+1}$. Since $w \in D_{k,n}$, we have $w_{i_1,i_2,\ldots,i_{n-k+1}} \in \Brun(T_{n-k+1})$. The following result can be proved along the lines of Proposition \ref{prop-D1}.

\begin{proposition}\label{prop-D2}
For $n \ge 4$, 
$$D_{k-1,n}= \big\{ w \prod_{1\leq i_1<i_2\cdots< i_{n-k+1}\leq n} (w_{i_1,i_2,\ldots,i_{n-k+1}}^{-1})^{c_{i_1,i_2,\ldots,i_{n-k+1}}} ~\mid~ w \in D_{k,n} \},$$
where $c_{i_1,i_2,\ldots,i_{n-k+1}} \in T_n$ is a coset representative of the permutation in $T_n/PT_n\cong S_n$ which takes $i_1,i_2,\ldots,i_{n-k+1}$ to $1,2,\ldots, n-k+1$, respectively, and fix everything else.
\end{proposition}

Beginning with $PT_n=D_{n-2,n}=D_{n-1,n}$ and iterating the procedure of constructing $D_{k-1,n}$ from $D_{k,n}$, we can construct all Brunnian twins on $n$ strands.

\begin{figure}[H]
\includegraphics[height=10cm]{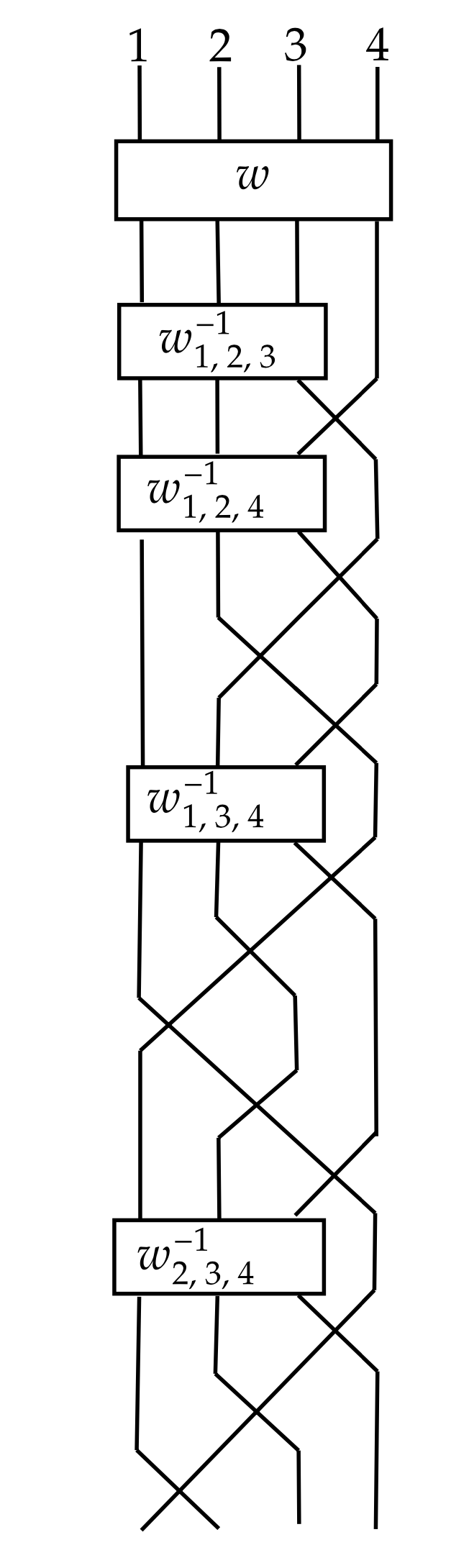}
\caption{Converting a pure twin into a Brunnian twin.}
\label{Decomposing a pure twin into Brunnian twins}
\end{figure}
\medskip

\subsection{Cohen twins}
Next, we consider another generalisation of Brunnian twins motivated by an idea due to Fred Cohen \cite{MR1349129}, and developed further for surface braid groups in \cite{MR3482589}. Recall that, for $0 \le i \le n-1$, the face map $d_i:T_n \to T_{n-1}$ deletes the $(i+1)$-st strand from the diagram of a twin. Although $d_i$ is not a group homomorphism,  it satisfies
\begin{equation}\label{di not homo on tn 2}
d_i(u w)= d_i(u)d_{\nu(u)(i+1)-1}(w),
\end{equation}
where $\nu:T_{n+1} \to S_{n+1}$ is the natural surjection.  For an arbitrary $u \in T_{n-1}$, we ask whether there exists $w \in T_n$ which is a solution of the system of equations
\begin{equation}\label{BTe2}
\left\{\begin{array}{l}
d_0(w)=u, \\
d_1(w)=u, \\
\vdots \\
d_{n-1}(w)=u.
\end{array}\right.
\end{equation}

Taking $u=1$ amounts to $w \in T_n$ being a Brunnian twin. 

\begin{definition}
A twin $w \in T_n$ is called a {\it Cohen twin} if   $d_0 (w)=d_1  (w)=\cdots=d_{n-1} (w)$. 
\end{definition}

For $n \ge 2$, let us set
$$
CT_n=\left\{w \in T_n ~\mid~ d_0 (w)=d_1  (w)=\cdots=d_{n-1} (w) \right\} .
$$
In other words, a twin on $n$ strands lie in $CT_n$ if it gives the same twin on $(n-1)$ strands after removing any one of its strands. For example, the twin $$\delta_n:=(t_1 t_2\cdots t_{n-1})(t_1 t_2\cdots t_{n-2})\cdots (t_1t_2)t_1$$ lies in $CT_n$ for all $n \geq 2$ and $d_0(\delta_n)=\delta_{n-1}$ (see Figure \ref{elements delta and gamma}).  Similarly, we define
$$
CPT_n = CT_n \cap PT_n=\left\{w \in PT_n ~\mid~ d_0 (w)=d_1 (w)=\cdots=d_{n-1} (w) \right\} .
$$
We refer to elements of $CPT_n$  as {\it pure Cohen twins}. For instance, the pure twin $$\gamma_n:=(t_1t_2\cdots t_{n-1})^n$$ lies in $CPT_n$ for all $n \geq 2$ and $d_0(\gamma_n)=\gamma_{n-1}$  (see Figure \ref{elements delta and gamma}). 
\begin{figure}[H]
\includegraphics[width=11cm]{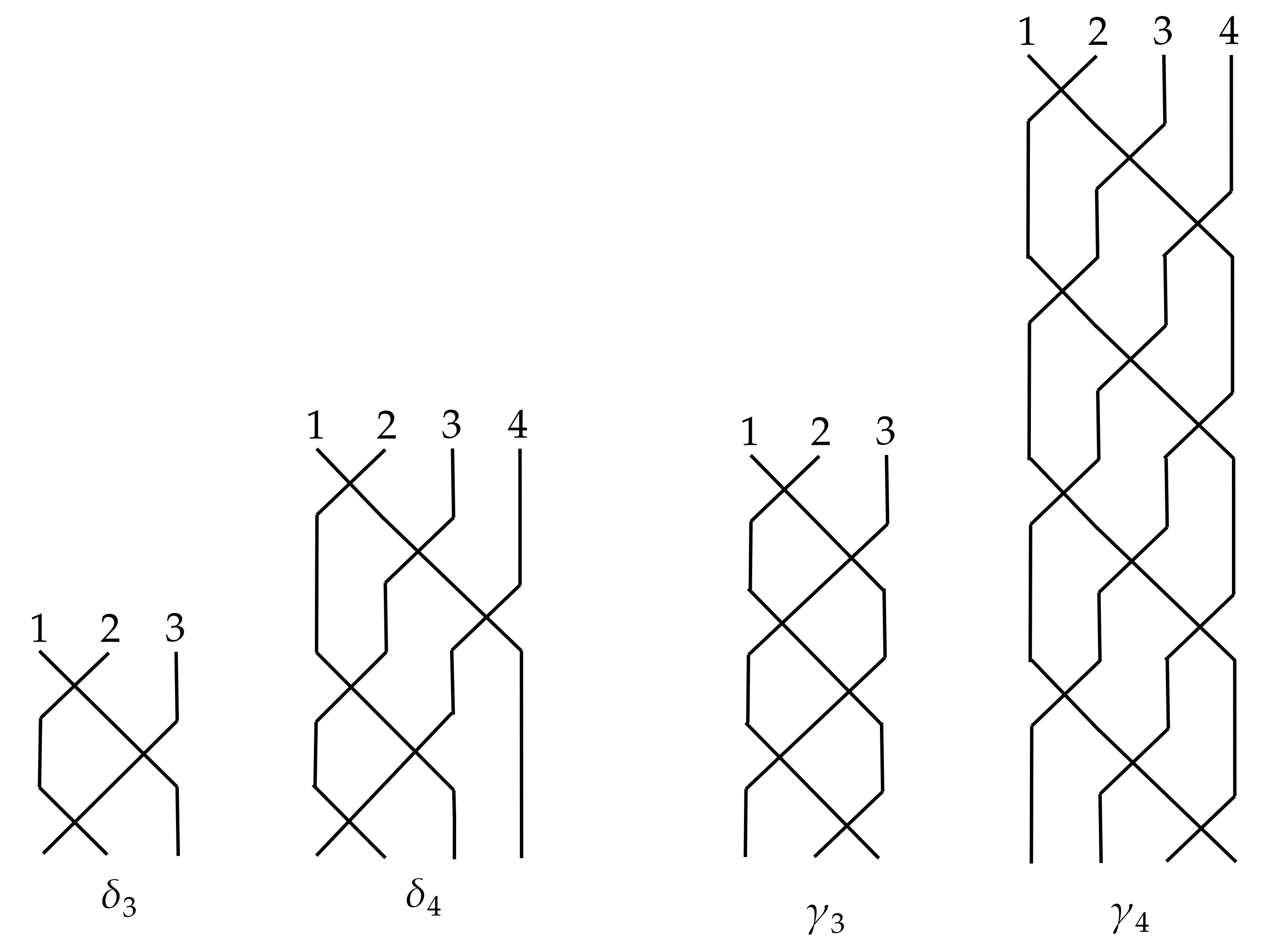}
\caption{Elements $\delta_3, \delta_4$ and $\gamma_3, \gamma_4$.} 
\label{elements delta and gamma}
\end{figure}

If $\phi, \psi: G \to H$ are group homomorphisms, then their {\it equalizer} is the subgroup of $G$ given by 
$$ \{g \in G ~\mid~ \phi(g)=\psi(g)\}.$$
Hence, $CPT_n$  is a subgroup of $PT_n$ being the equalizer of group homomorphisms $d_0, d_1, \ldots, d_{n-1}:PT_n\to PT_{n-1}$.

\begin{proposition}\label{homo from cptn to cptn-1}
The following assertions hold:
\begin{enumerate}
\item For each $0 \le i \le n-1$, $d_i(CPT_n) \subseteq CPT_{n-1}$ and the map  $d_0=d_1=\cdots=d_{n-1}:CPT_n\to CPT_{n-1}$ is a group homomorphism.
\item The set $CT_n$ is a subgroup of $T_n$. Moreover, for each $0 \le i \le n-1$, $d_i(CT_n) \subseteq CT_{n-1}$ and the map  $d_0=d_1=\cdots=d_{n-1}:CT_n\to CT_{n-1}$ is a group homomorphism.
\end{enumerate}
\end{proposition}

\begin{proof}
Let $w \in CPT_n$ and $0 \le i \le n-1$. Then, using  \eqref{face map identity}, we obtain
\begin{equation}\label{proof equation 1}
d_j (d_i (w) )=d_j (d_0 (w) )=d_0 (d_{j+1} (w) )=d_0 (d_i (w) )
\end{equation}
for each $0 \le j \le n-2$, and hence  $d_i(CPT_n) \subseteq CPT_{n-1}$. That $d_0=d_1=\cdots=d_{n-1}:CPT_n\to CPT_{n-1}$ is a group homomorphism follows from Proposition \ref{bi-delta-set on tn}. 
\par 
For the second assertion, let $u, w \in CT_n$. By \eqref{di not homo on tn 2}, we have
\begin{equation}\label{pe1}
d_i(u w)=d_i(u) d_{\nu(u)(i+1)-1}(w)=d_0(u) d_{\nu(u)(1)-1}(w)=d_0(uw)
\end{equation}
for each $0 \leq i \leq n-1$, and hence $u w \in CT_n$. Further, the equation
$$
1=d_i\left(u^{-1} u\right)=d_i\left(u^{-1}\right) d_{\nu(u^{-1})(i+1)-1}(u)=d_i\left(u^{-1}\right) d_0(u),
$$
gives
$$
d_i\left(u^{-1}\right)=\left(d_0(u)\right)^{-1}
$$
for each $0 \leq i \leq n-1$, and hence $CT_n$ is a subgroup of $T_n$. The proof of $d_i(CT_n) \subseteq CT_{n-1}$  follows from \ref{proof equation 1}. Finally,  \eqref{pe1} also shows that $d_0=d_1=\cdots=d_{n-1}: CT_n \rightarrow CT_{n-1}$ is a group homomorphism.
\end{proof}

\begin{proposition}\label{CPTn<CTn}
$CPT_n$ is an index two subgroup of $CT_n$ for $n \ge 3$.
\end{proposition}

\begin{proof}
The topological interpretation of elements of $T_n$ can be applied to elements of $S_n$ as well by allowing triple intersection points. Thus, for each $0 \leq i \leq n-1$, there is a map $\bar{d_i}: S_n \to S_{n-1}$ (thought of as deleting the $(i+1)$-st strand) such the following diagram commutes
$$
\begin{tikzcd}
	{PT_n} & {T_n} & {S_n} \\
	{PT_{n-1}} & {T_{n-1}} & {S_{n-1}}
	\arrow[hook, from=1-1, to=1-2]
	\arrow[hook, from=2-1, to=2-2]
	\arrow["{\nu_n}",two heads, from=1-2, to=1-3]
	\arrow["{\nu_{n-1}}",two heads, from=2-2, to=2-3]
	\arrow["{d_i}", from=1-1, to=2-1]
	\arrow["{d_i}", from=1-2, to=2-2]
	\arrow["{\bar{d_i}}", from=1-3, to=2-3].
\end{tikzcd}
$$
Set $CS_n:=\nu_n(CT_n)$ for each $n \ge 2$. Note that $CS_2=\nu_2(T_2)=S_2 \cong \mathbb Z_2$. The commutativity of the preceding diagram shows that every $\tau \in CS_n$ satisfy $\bar{d_0}(\tau)=\bar{d_1}(\tau)= \cdots= \bar{d}_{n-1}(\tau)$. By Proposition \ref{homo from cptn to cptn-1}(2), we have $d_0(CT_n)\subseteq CT_{n-1}$. The commutativity of the preceding diagram implies that $\bar{d_0}(CS_n)= \bar{d_0}\nu_n(CT_n)= \nu_{n-1}d_0(CT_n) \subseteq \nu_{n-1}(CT_{n-1})= CS_{n-1}$.  Thus, for $n\geq 3$,  the restriction of the map $\bar{d_0}:S_n \to S_{n-1}$  induces a map $\bar{d_0} :CS_n  \to CS_{n-1}$ such that $\ker(\bar{d_0})=\cap_{i=0}^{n-1} \ker(\bar{d_i})$. Direct computation gives $\ker(\bar{d_0})=1$, and hence the map $\bar{d_0}\cdots \bar{d_0}:CS_n\to CS_2$ is injective. Since $\nu_n(\delta_n)\neq 1$, we have $CT_n/CPT_n \cong CS_n\cong \mathbb Z_2$, and the proof is complete. 
\end{proof}

The following result follows along the lines of \cite[Lemma 2.10]{MR1955357}.

\begin{proposition}\label{surcptn}
For each $1 \leq k \leq n-1$, the map $$\underbrace{d_0 \cdots d_0}_{(n-k)~ \text{times}}:CPT_n \to CPT_{k}$$ is surjective. In particular, 
 $d_0:CPT_n \to CPT_{n-1}$ is surjective for $n \ge 2$.
\end{proposition}

\begin{proof}
Let us set $d_{n-k,n}=\underbrace{d_0 \cdots d_0}_{(n-k)~ \text{times}}$. We use induction on $k$. Clearly, for $k=1$, the map $d_{n-1,n}:CPT_n\to CPT_{1}$ is surjective. Assume that $d_{n-k+1, n}$ is surjective with $k>1$, and let $w \in CPT_k$.
\par 
Case 1: Suppose that $w \in \ker(d_0:CPT_k \to CPT_{k-1})$. Then consider the element
$$
w_{k, n}=\prod_{0 \leq i_1<i_2<\cdots<i_{n-k} \leq n-1} d^{i_{n-k}} d^{i_{n-k-1}} \cdots d^{i_1} (w)
$$
of  $PT_n$ with lexicographic order on the indices from the right. Since $w \in \ker(d_0:CPT_k \to CPT_{k-1})$, a straightforward computation shows that $w_{k,n} \in CPT_n$ and $d_{n-k,n}(w_{n,k})=w$. For instance, taking $n=4$ and $k=1$, we have
$$
w_{1, 4}=\prod_{0 \leq i_1<i_2<i_{3} \leq 3} d^{i_{3}} d^{i_{2}} d^{i_1} (w)
$$
with lexicographic order from the right. Note that $(i_1, i_2, i_3)\in \{(0,1,2), (0,1,3), (0,2,3), (1,2,3)\}$ and  $w_{1,4} = d^2d^1d^0(w)~ d^3d^1d^0(w)~ d^3d^2d^0(w)~ d^3d^2d^1(w).$ Direct computations give
\begin{eqnarray*}
d_0(w_{1,4}) &=& d^1d^0(w)~ d^2d^0(w)~ d^2d^1(w)~ d^2d^1d^0(d_0(w)),\\
d_1(w_{1,4}) &=& d^1d^0(w)~ d^2d^0(w)~ d^2d^1d^0(d_0(w))~ d^2d^1(w),\\
d_2(w_{1,4}) &=& d^1d^0(w)~ d^2d^1d^0(d_0(w))~ d^2d^0(w)~ d^2d^1(w),\\
d_3(w_{1,4}) &=& d^2d^1d^0(d_0(w))~ d^1d^0(w)~ d^2d^0(w)~ d^2d^1(w).
\end{eqnarray*}
Since $w \in \ker(d_0:CPT_k \to CPT_{k-1})$, $d^2d^1d^0(d_0(w))=1$, and hence $w_{1,4}\in CPT_{4}$.

\par
Case 2: Now, suppose that $1 \ne \delta = d_0(w) \in CPT_{k-1}$. By induction hypothesis, there exists $\gamma\in CPT_n$ such that $d_{n-k+1, n}(\gamma)= d_0(d_{n-k,n}(\gamma))=\delta$. Note that $$w \, d_{n-k,n}(\gamma)^{-1} \in \ker(d_0:CPT_k \to CPT_{k-1}).$$
 Thus, by  Case 1, there exists $\lambda \in CPT_n$ such that $$d_{n-k,n}(\lambda)=w\, d_{n-k,n}(\gamma)^{-1},$$ and hence  $d_{n-k,n}(\lambda \gamma) = w$. This proves that the map $d_{n-k, n}$ is surjective. 
\end{proof}

\begin{proposition}
The map $d_0:CT_n \to CT_{n-1}$ is surjective for each $n \ge 2$.
\end{proposition}
\begin{proof}
In view of  Proposition \ref{CPTn<CTn}, we can write $CT_{n-1}= CPT_{n-1} \cup \delta_{n-1}CPT_{n-1}$. Let us take $w \in CT_{n-1}$. If $w \in CPT_{n-1}$,  then by Proposition \ref{surcptn}, there exists an $u \in CPT_n$ such that $d_0(u)=w$. If $w \in  \delta_{n-1}CPT_{n-1}$, then again by Proposition \ref{surcptn}, there exists $v \in CPT_n$, such that $d_0(v)=\delta_{n-1}^{-1}w$, and hence $d_{0}(\delta_n v)=w$. This complete the proof.  
\end{proof}

Thus, we obtain the following short exact sequences
$$
1\to \Brun(T_n) \to CT_n \to CT_{n-1} \to 1
$$
and 
$$
1\to \Brun(T_n) \to CPT_n \to CPT_{n-1} \to 1.
$$
Observe that $CPT_2=\Brun(T_2)=PT_2=1$ and $CPT_3=\Brun(T_3)=PT_3=\langle (t_1t_2)^3\rangle \cong \mathbb{Z}$. Thus, the preceding exact sequence gives $CPT_4=\Brun(T_4)\rtimes \langle (t_1t_2)^3 \rangle.$

\begin{theorem}\label{P1pure}
For each $u \in PT_{n-1}$ or $u \in T_{n-1}$, the system of equations
\begin{equation}
\left\{\begin{array}{l}
d_0(w)=u, \\
d_1(w)=u, \\
\vdots \\
d_{n-1}(w)=u,
\end{array}\right.
\end{equation} 
has a solution if and only if $u$ satisfies the condition
$$d_0(u)=d_1(u)=\cdots=d_{n-2}(u).$$
\end{theorem}

\begin{proof}
Let $u \in PT_{n-1}$ such that  the system of equations \eqref{BTe2} has a solution. Then there exists $w \in PT_n$ such that $d_0 (w)=\cdots=d_{n-1} (w)=u$. It follows from Proposition \ref{homo from cptn to cptn-1} that $u \in CPT_{n-1}$, and hence $d_0( u)=\cdots=d_{n-2}(u)$. Conversely, suppose that $d_0 (u)=\cdots=d_{n-2} (u)$, that is, $u \in CPT_{n-1}$. By Proposition \ref{surcptn}, $d_0:CPT_n \to CPT_{n-1}$ is surjective, and hence there exists $w \in CPT_n$ which is a solution  to \eqref{BTe2}. The proof for the case when $u \in T_{n-1}$ is similar.
\end{proof}
\medskip

\section{Brunnian doodles on the 2-sphere}\label{section 5}

Note that the closure of a Brunnian braid is a Brunnian link. The converse is not true and there exist Brunnian links that cannot be obtained as the closure of Brunnian braids (see \cite{MR3200492}). The same scenario occurs with doodles on the 2-sphere. Consider the Brunnian doodle on the 2-sphere as shown in Figure \ref{A Brunnian doodle}. The main result of this section will show that this Brunnian doodle cannot be realised as the closure of a Brunnian twin. 
\begin{figure}[H]
\includegraphics[height=3.8cm]{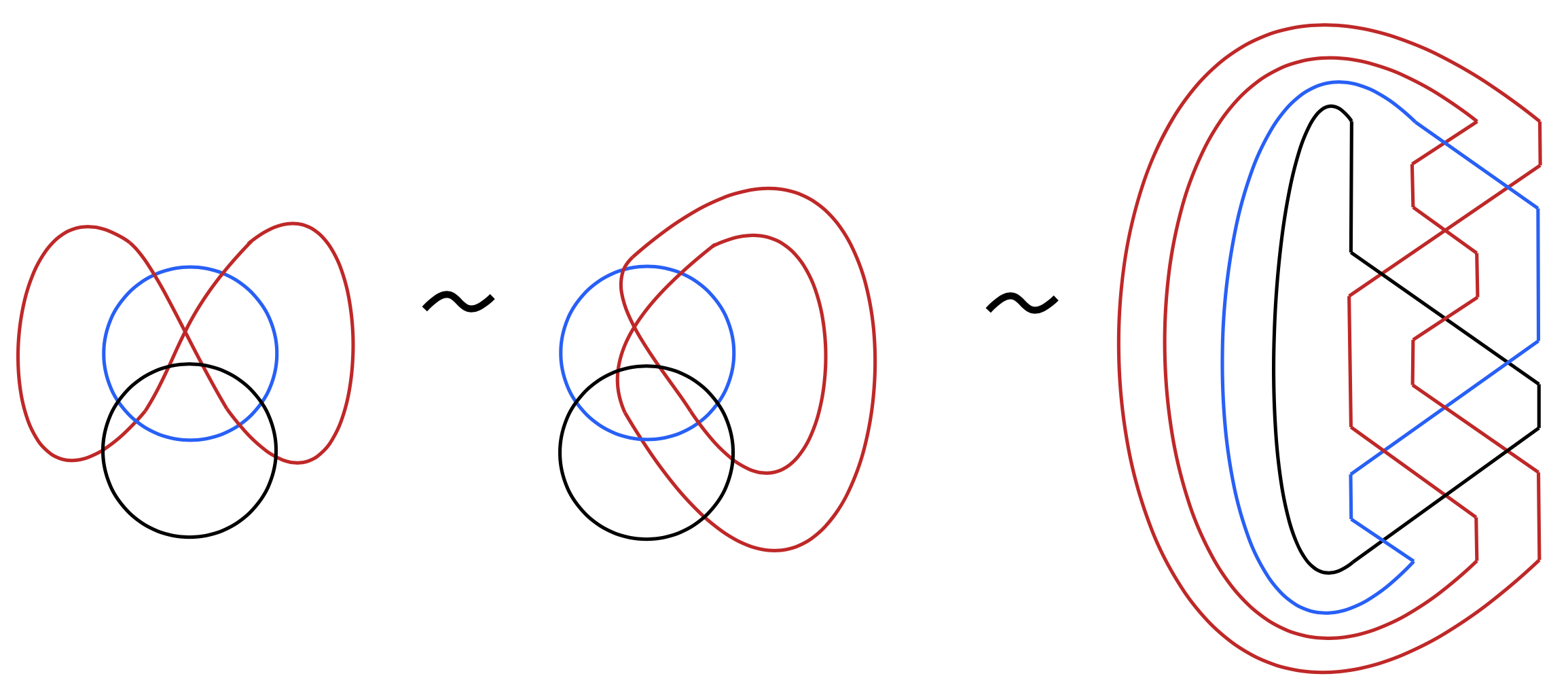}
\caption{A Brunnian doodle which is not the closure of a Brunnian twin.}
\label{A Brunnian doodle}
\end{figure}

\begin{definition}
A doodle diagram on the 2-sphere is called minimal if it has no monogons and bigons.
\end{definition}

\begin{theorem}\cite[Theorem 2.2]{MR1370644}\label{T2.2}
Any doodle has a unique (up to the transformation shown in Figure \ref{doodle move}) minimal doodle diagram with a minimal number of intersection points. Further, this minimal doodle diagram can be constructed from any other doodle diagram by applying Reidemeister moves $R1$ and $R2$ that reduce the number of intersection points.
\end{theorem}

\begin{figure}[H]
\includegraphics[height=3.2cm]{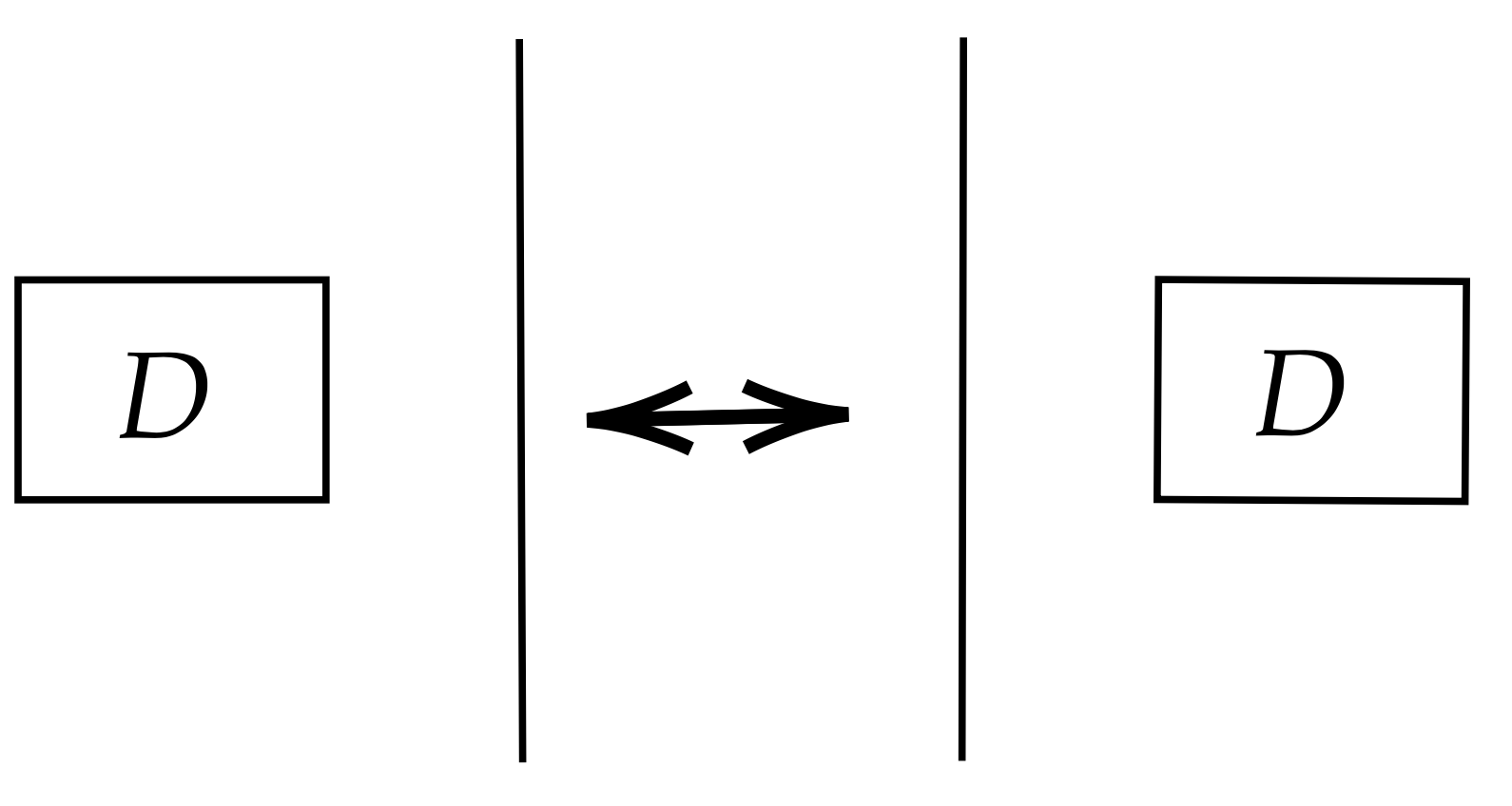}
\caption{Transformation of doodle diagrams.}
\label{doodle move}
\end{figure}

\begin{figure}[H]
\includegraphics[height=3.2cm]{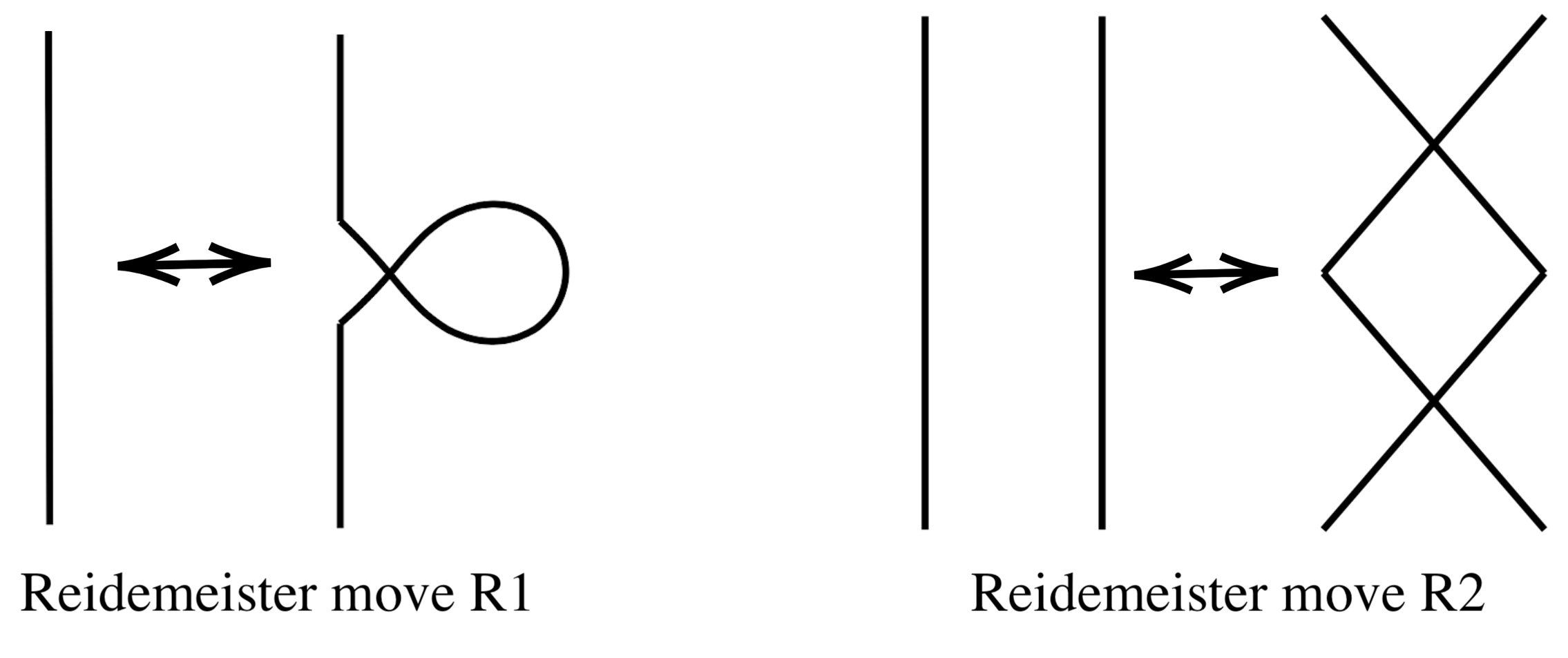}
\caption{Reidemeister moves.}
\label{doodle move}
\end{figure}

For a given reduced word $w=t_{i_1} \dots t_{i_k} \in T_n$, let $\ell(w)=k$ be the \textit{length} of $w$.  For each $1\leq i\leq n-1$, if $\log_{i}(w)$ denote the number of $t_i$'s present in the expression $w$, then
 $$\ell(w) =\sum_{i=1}^{n-1}\log_{i}(w).$$ 

A \textit{cyclic permutation} of a word  $w=t_{i_1}\dots t_{i_k} \in T_n$ (not necessarily reduced) is a word $w'=t_{i_r}t_{i_{r+1}} \dots t_{i_k}t_{i_1} t_{i_2}\cdots t_{i_{r-1}}$ for some $1\leq r\leq k$. It is easy to see that $w$ and $w'$ are conjugate to each other in $T_n$, in fact, $w'=(t_{i_1} t_{i_2}\dots t_{i_{r-1}})^{-1}w(t_{i_1} t_{i_2}\dots t_{i_{r-1}})$. A word $w$ is called \textit{cyclically reduced} if each cyclic permutation of $w$ is reduced. Clearly, a cyclically reduced word is reduced.

\begin{lemma}\label{BDL1}
Let $w \in PT_n$ be a pure twin. Then the following assertions hold:
\begin{enumerate}
\item If $\ell(w)$ is minimal among all the elements in the conjugacy class of $w$, then the closure of $w$ is a minimal doodle diagram.
\item The closure of $w$ is an $n$-component trivial doodle if and only if $w$ is a trivial twin.
 \end{enumerate}
\end{lemma}

\begin{proof}
It follows from \cite[Corollary 2.4]{MR4145210} that each word in $T_n$ is conjugate to some cyclically reduced word. Since $\ell(w)$ is minimal among all the elements in the conjugacy class of $w$, it follows that $w$ is a cyclically reduced word. Hence, the closure of $w$ has no bigons. Since $w$ is pure twin, its closure has no monogons, and hence the diagram is minimal.
\par
By Markov Theorem for doodles on the 2-sphere \cite[Theorem 4.1]{Gotin}, conjugate twins have the same closure. Thus, we can assume that $\ell(w)$ is minimal among all the elements in the conjugacy class of $w$. It follows from assertion (1) that the closure of $w$ is a minimal doodle diagram. Note that the number of double points in the closure of the twin $w$ equals $\ell(w)$, and hence $\ell(w)=0$. But, this implies that $w$ is trivial twin. The converse implication in assertion (2) is obvious.
\end{proof}

Let $\widehat{w}$ denote the closure of a twin $w$ on the 2-sphere. By \cite[Theorem 2.1]{MR1370644}, every oriented doodle on the 2-sphere is the closure of a twin. The {\it twin index} $I(D)$ of a doodle $D$ on the 2-sphere is the minimal $n$ such that there is a twin $w \in T_n$ whose closure is equivalent to $D$.

\begin{theorem}\label{brunnian doodle twin index}
 An $m$-component Brunnian doodle $D$ on the 2-sphere is the closure of a Brunnian twin if and only if $I(D) = m$.
\end{theorem}
\begin{proof}
 If $u$ is a Brunnian twin on $m$ strands, then its closure on the 2-sphere is a Brunnian doodle on $m$ components with $I(\widehat{u})=m$. Conversely, if $D$ is a Brunnian doodle on $m$ components and $I(D)=m$, then there exist $w \in PT_m$ such that $\widehat{w}=D$. Removing any strand from $w$ corresponds to removing a component from $D$. Thus, $\widehat{d_i(w)}$ is a trivial doodle for each $i$. By Lemma \ref{BDL1}, $d_i(w)=1$ for each $i$, and hence $w$ is a Brunnian twin.
\end{proof}

\begin{remark}
An analogue of Theorem \ref{brunnian doodle twin index} for Brunnian links in $S^3$ is proved in \cite[Theorem 2.2]{MR1822143}.
\end{remark}
\medskip


\section{Simplicial  structure on pure twin groups}\label{section 6}
In this section, we discuss simplicial structures on twin and pure twin groups and relate them with Milnor's construction for simplicial spheres.

\subsection{Simplicial sets and simplicial groups}

We recall some basic definitions and constructions \cite{MR0222892, Milnor}.
\begin{definition}
 A sequence of sets $X_* = \{ X_n \}_{n \geq 0}$  is called a {\it simplicial set} if there are face maps
$$
d_i : X_n \longrightarrow X_{n-1} ~\mbox{for}~0 \leq i \leq n
$$
and degeneracy  maps
$$
s_i : X_n \longrightarrow X_{n+1} ~\mbox{for}~0 \leq i \leq n,
$$
which satisfy the following simplicial identities:
\begin{enumerate}
\item $d_i d_j = d_{j-1} d_i$ if $i < j$,
\item $s_i s_j = s_{j+1} s_i$ if $i \leq j$,
\item $d_i s_j = s_{j-1} d_i$ if $i < j$,
\item $d_j s_j = \id = d_{j+1} s_j$,
\item $d_i s_j = s_{j} d_{i-1}$ if $i > j+1$.
\end{enumerate}
\end{definition}

We view $X_n$ geometrically as the set of $n$-simplices including all possible degenerate simplices. Here, a simplex $x$ is {\it degenerate} if $x = s_i (y)$ for some simplex $y$ and degeneracy operator $s_i$, otherwise $x$ is {\it non-degenerate}.  A simplicial set $X_*$ is \textit{pointed} if we fix a basepoint $\star \in X_0$ that creates one and only one degenerate $n$-simplex in each $X_n$ by applying iterated degeneracy operations on it. A \textit{simplicial group}  is a simplicial set $X_*$ such that each $X_n$ is a group and all face and degeneracy maps are group homomorphisms.
 
\begin{remark} 
In the context of braid-type groups (for example, braid group $B_n$, virtual braid group $VB_n$, welded braid group $WB_n$, etc.), the maps $d_i$ usually represents deleting of the $(i+1)$-th strand and $s_i$ represents doubling of the $(i+1)$-th strand.
\end{remark}

\begin{remark}
Note that the defining identities  of a bi-$\Delta$-set and that of a simplicial set are  similar. The only differences are that we don't have  $d_{j+1} s_j= \id$ for bi-$\Delta$-sets, and when viewed as maps from $X_{n- 1} \to X_n$, the number of degeneracy maps is one less than the number of coface maps. We have used the bi-$\Delta$-set structure at three instances in the preceding sections. The first instance of usage of a bi-$\Delta$-set is Proposition \ref{bi-delta-set on tn}, though its arguments can be modified to adapt to a simplicial set structure. The second instance is the proof of Proposition \ref{surcptn}, where we defined the element $w_{k,n}$ and showed that $w_{k,n}\in CPT_n$. In the latter case, a simplicial structure would not be helpful. Finally, using the Decomposition Theorem for  bi-$\Delta$-groups, we have given a decomposition of pure twin groups in Proposition \ref{prop:ptndecomposition} with Brunnian subgroups as constituents.
\end{remark}

Let $G_* = \{ G_n \}_{n \geq 0}$ be a simplicial group. The group of \textit{Moore $n$-cycles} $Z_n(G_*)\leq G_n$ is defined by
$$
Z_n(G_*)=\bigcap_{i=0}^n\mathrm{Ker}(d_i\colon G_n\to G_{n-1})
$$
and the group of \textit{Moore $n$-boundaries} $B_n(G_*)\leq G_n$ is defined by
$$
B_n(G_*)=d_0\left(\bigcap_{i=1}^{n+1}\mathrm{Ker}(d_i\colon G_{n+1}\to G_n)\right).
$$
Simplicial identities guarantees that $B_n(G_*)$ is a (normal) subgroup of $Z_n(G_*)$ (see \cite[Proposition 4.1.3]{MR2188127} or \cite[Example 7.7]{MR2915498}). The $n$-th \textit{Moore homotopy group} $\pi_n(G_*)$ of $G_*$ is defined by
$$
\pi_n(G_*)=Z_n(G_*)/B_n(G_*).
$$
It is a classical result due to Moore \cite{moore} that $\pi_n(G_*) \cong \pi_n(|G_*|)$, where $|G_*|$ is the geometric realisation of $G_*$. A simplicial group $G_*$ is called {\it contractible} if $\pi_n (G_*) = 1$ for all $n>0$.

\medskip

 Milnor's $F[K]$ construction is the adjoint functor to the forgetful functor from the category of pointed simplicial groups to the category of pointed simplicial sets. For a given pointed simplicial set $K_* = \{ K_n, \star \}_{n \geq 0}$, Milnor's $F[K]$ construction is the simplicial group with $F[K]_n = F(K_n \setminus \star)$, the free group on $K_n \setminus \star$, with the face and the degeneracy maps induced from the face and degeneracy maps of $K_*$. It is well-known from \cite{Milnor} that there is weak homotopy equivalence
\begin{equation}\label{milnor weak homotopy}
|F[K]_*| \simeq \Omega\Sigma|K_*|,
\end{equation}
where $|X_*|$ denotes the geometric realisation of a simplicial set $X_*$. Here, $\Omega Z$ is the loop space of all based loops in a pointed topological space $Z$ and $\Sigma Z$  is the reduced suspension of $Z$.
\par

Consider the pointed simplicial 2-sphere $S^2 = \Delta[2] / \partial \Delta[2]$ with
$$
S^2_0 = \{ \star \},~~S^2_1 = \{ \star \},~~S^2_2 = \{ \star, \sigma \},~~S^2_3 = \{ \star, s_0 (\sigma),  s_1(\sigma), s_2 (\sigma)\}, \ldots,  S^2_n = \{ \star, x_{ij}~\mid~0 \leq i < j \leq n-1 \},
\ldots
$$
where $\sigma = (0,1,2)$ is the non-degenerate 2-simplex, $x_{ij} = s_{n-1} \ldots s_{j+1} \widehat{s_j} s_{j-1} \ldots s_{i+1} \widehat{s_i} s_{i-1} \ldots s_0 (\sigma)$ and $\widehat{s_k}$ means that the degeneracy map $s_k$ is omitted. Then $F[S^2]$ construction has the following terms:
\begin{align*}
F[S^2]_0 &= 1,\\
F[S^2]_1 & = 1, \\
F[S^2]_2 & = F(\sigma),\\
F[S^2]_3 & = F(s_0 (\sigma),  s_1 (\sigma), s_2 (\sigma)),\\
F[S^2]_4 & = F(s_1s_0 (\sigma),  s_2s_0 (\sigma), s_3s_0 (\sigma), s_2s_1 (\sigma), s_3s_1 (\sigma), s_3s_2 (\sigma)),\\
\vdots & \\
F[S^2]_n & = F(x_{ij};~~0 \leq i < j \leq n-1),\\
\vdots &
\end{align*}

For each $n \ge 2$, the group $F[S^2]_n$ is a free group of rank $n(n-1)/2$. In this construction of the simplicial 2-sphere, it is convenient to present the degeneracy map $s_i$ is a doubling of the $(i+1)$-th component and the face map $d_i$ as deletion of the $(i+1)$-th component. For example,
$$
s_0 (\sigma) = (0,0,1,2),\quad s_1 (\sigma) = (0,1,1,2), \quad s_2 (\sigma) = (0,1,2,2),
$$
$$
s_1 s_0 (\sigma) = (0,0,0,1,2), \quad s_2 s_0 (\sigma) = (0,0,1,1,2), \quad s_3 s_0 (\sigma) = (0,0,1,2,2),
$$
$$
s_2 s_1 (\sigma) = (0,1,1,1,2), \quad s_3 s_1 (\sigma) = (0,1,1,2,2), \quad s_3 s_2 (\sigma) = (0,1,2,2,2).
$$
The face and degeneracy maps are determined with respect to the standard simplicial identities for simplicial groups. For example, the first non-trivial face maps $d_i : F[S^2]_3 \to F[S^2]_2$ are given by
\begin{align*}
d_0 : s_0 (\sigma)  \mapsto \sigma, \quad & s_1 (\sigma)  \mapsto \star,  \quad s_2(\sigma)  \mapsto \star, \\
d_1 : s_0 (\sigma)  \mapsto \sigma, \quad  & s_1 (\sigma)  \mapsto \sigma,  \quad s_2 (\sigma)  \mapsto \star, \\
d_2 : s_0 (\sigma)  \mapsto \star, \quad  & s_1 (\sigma)  \mapsto \sigma,  \quad s_2 (\sigma)  \mapsto \sigma, \\
d_3 : s_0 (\sigma)  \mapsto \star, \quad  & s_1 (\sigma)  \mapsto \star, \quad s_2(\sigma)  \mapsto \sigma.
\end{align*}

Milnor's construction gives a possibility to define the homotopy groups $\pi_n(S^3)$ combinatorially, in terms of free groups. By \eqref{milnor weak homotopy}, the geometric realisation of $F[S^2]_*$ is weakly homotopically equivalent to the loop space $\Omega S^3$. Thus, the homotopy groups of $S^3$ are isomorphic to the Moore homotopy groups of $F[S^2]$, that is,
\begin{equation}\label{pi1s3 and moore homotopy}
\pi_{n+1}(S^3) \cong Z_n (F[S^2]_*) / B_n (F[S^2]_*).
\end{equation}
\medskip


\subsection{Simplicial pure twin group} \label{sec2}
By \cite[Theorem 2]{MR4027588}, we have $PT_3 = \langle (t_1 t_2)^3 \rangle \cong \mathbb{Z}$ and $PT_4 \cong F_7$, where $F_7$ is the free group on the elements
$$
x_1 = (t_1 t_2)^3,\quad x_2 = \left( (t_1 t_2)^3 \right)^{t_3}, \quad x_3 = \big( (t_1 t_2)^3\big)^{t_3 t_2},\quad x_4 = \left( (t_1 t_2)^3 \right)^{t_3 t_2 t_1},$$
$$ \quad x_5 = (t_2 t_3)^3,  \quad  x_6 = \left((t_2 t_3)^3 \right)^{t_1}, \quad x_7 = \big((t_2 t_3)^3\big)^{t_1 t_2}.
$$

 Let $SPT_*= \{SPT_n\}_{n \ge 0}$, where $SPT_n=PT_{n+1}$ for each $n\geq0$. Following the methodology of~\cite{MR2853222}, consider the sequence of groups
$$
 \ldots\ \begin{matrix}\longrightarrow\\[-3.5mm] \ldots\\[-2.5mm]\longrightarrow\\[-3.5mm]
\longleftarrow\\[-3.5mm]\ldots\\[-2.5mm]\longleftarrow \end{matrix}\ PT_4 \ \begin{matrix}\longrightarrow\\[-3.5mm]\longrightarrow\\[-3.5mm]\longrightarrow\\[-3.5mm]\longrightarrow\\[-3.5mm]\longleftarrow\\[-3.5mm]
\longleftarrow\\[-3.5mm]\longleftarrow
\end{matrix}\ PT_3\ \begin{matrix}\longrightarrow\\[-3.5mm] \longrightarrow\\[-3.5mm]\longrightarrow\\[-3.5mm]
\longleftarrow\\[-3.5mm]\longleftarrow \end{matrix}\ PT_2\ \begin{matrix} \longrightarrow\\[-3.5mm]\longrightarrow\\[-3.5mm]
\longleftarrow \end{matrix}\ PT_1
$$
with face and degeneracy homomorphisms
\begin{align*}
& d_i: SPT_n=PT_{n+1}\to SPT_{n-1}=PT_n,\\
& s_i: SPT_n=PT_{n+1}\to SPT_{n+1}=PT_{n+2},\
\end{align*}
where the face map $d_i$ is the deleting of the $(i+1)$-th strand and the degeneracy map $s_i$ is the doubling of the $(i+1)$-th strand for each $0 \le i \le n$. For example, we prove in the proof of Proposition \ref{brun t4} that $d_3 : PT_4 \to PT_3$ is given by
$$
d_3(x_1) = y \quad \textrm{and} \quad d_3(x_2) = d_3(x_3) = d_3(x_4) = d_3(x_5) = d_3(x_6) = d_3(x_7) =1,
$$
where $y=(t_1 t_2)^3 \in PT_3$. As in the classical case it is not difficult to prove the following result, whose proof is adapted from \cite[Proposition 3.1]{BW2}.

\begin{proposition}
$SPT_*$ is a contractible simplicial group.
\end{proposition}

\begin{proof}
Let $x \in Z_n(SPT_*)$ be a Moore $n$-cycle, that is, $x \in SPT_n$ and $d_i(x) = 1$ for all $0 \le i \le n$. Note that $SPT_*$ admits an additional degeneracy map $\iota_{n+1}: SPT_{n} \rightarrow SPT_{n+1}$, which adds a trivial strand on the left of the diagram of the twin. If we set $y = \iota_{n+1}(x) \in SPT_{n+1}$, then we see that
$d_j(y) = 1$ for all $1 \le j \le n+1$ and $d_0(y) = x$. Thus, $x \in B_n(SPT_*)$ is a Moore $n$-boundary, and hence $\pi_n(SPT_*) = 1$ for all $n$.
\end{proof}

We write $U_{n,i} := \mathrm{Ker}(d_i : PT_n \to PT_{n-1})$ for each $0 \le i \le n-1$. Then, we have the following short exact sequence
$$
\begin{tikzcd}
	1 & {U_{n,i}} & {PT_n} & {PT_{n-1}} & 1
	\arrow[from=1-1, to=1-2]
	\arrow[from=1-2, to=1-3]
	\arrow["{d_i}", from=1-3, to=1-4]
	\arrow[from=1-4, to=1-5]
\end{tikzcd}
$$
with the splitting given by $d^i: PT_{n-1}\to PT_n$ as defined in Proposition \ref{bi-delta-set on tn}. This gives a  semi-direct product  decomposition $PT_n = U_{n,i} \rtimes PT_{n-1}$. Clearly, $U_{3,0}=U_{3,1}=U_{3,2}=PT_3$. The following problem seems interesting.

\begin{problem}
Find presentations of $U_{n,i}$ for $n \geq 4$.
\end{problem}

We construct a simplicial subgroup $K_*$ of $SPT_*$ which would be the image of the simplicial sphere $S^2$ under a simplicial map. Put $K_0 = K_1 = 1$, $K_2 = SPT_2 = \langle c_{111} \rangle$, the infinite cyclic group generated by $c_{111}= (t_1 t_2)^3$, and
$$
K_3 =  \langle c_{211} = s_0 (c_{111}),~~ c_{121} = s_1 (c_{111}),~~c_{112} = s_2 (c_{111}) \rangle.
$$
 In general, we define
$$
K_n = \langle c_{klm} = s_{n-1} \ldots s_{j+1} \widehat{s_j} s_{j-1} \ldots s_{i+1} \widehat{s_i} s_{i-1} \ldots s_0 (c_{111})~\mid~0 \leq i < j \leq n-1,~~k+l+m = n+1 \rangle,
$$
the subgroup of $SPT_n$ generated by $n(n-1)/2$ elements. It follows from the simplicial identities that $d_i(c_{klm})\in K_{n-1}$ and  $s_j(c_{klm})\in K_{n+1}$ for each generator $c_{klm}$ of $K_n$ and all $d_i, s_j$. Thus,
for each $n \ge 0$, restriction of face maps $d_i: SPT_n \to SPT_{n-1}$ gives face maps $d_i:K_n \to K_{n-1}$. Similarly, restriction of degeneracy maps $s_i: SPT_n \to SPT_{n+1}$ induce degeneracy maps $s_i: K_n \to K_{n+1}$, turning $K_*=\{K_n\}_{n \ge 0}$ into a simplicial subgroup of $SPT_*$. 

\begin{theorem}\label{k3 and k4}
$K_3 \cong F[S^2]_3$ and  $K_4 \cong F[S^2]_4$.
\end{theorem}

\begin{proof}
Using the geometrical interpretation of $c_{111}$ (see Figure \ref{pure twin in pt3}) and degeneracy maps $s_i$, we write the generators of $K_3$ in terms of the generators of $PT_4$ as follows:
$$
c_{211} = (t_2 t_1 t_3 t_2 ) (t_1 t_2 t_3 )( t_1 t_2 t_3) = \left( (t_1 t_2)^3 \right)^{t_3 t_2}  (t_2 t_3)^3 = x_3 x_5,
$$
$$
c_{121} = (t_1 t_2 t_3 )( t_2 t_1 t_3 t_2 )( t_1 t_2 t_3) = \left( (t_2 t_3)^3 \right)^{t_1}  \left( (t_1 t_2)^3 \right)^{t_3} = x_6 x_2,
$$
$$
c_{112} =  (t_1 t_2 t_3 )( t_1 t_2 t_3 )( t_2 t_1 t_3 t_2) = (t_1 t_2)^3 \left( (t_2 t_3)^3 \right)^{t_1 t_2}  = x_1 x_7.
$$
Since $PT_4$ is a free group of rank 7, it follows that $K_3$ is a free group of rank 3, and hence $K_3 \cong F[S^2]_3$. It is known from \cite[Section 5]{MR4170471} that $SPT_4 = PT_5$ is free group of rank 31, but \cite{MR4170471} does not give any free generating set for $PT_5$. However, using \cite[Theorem 4]{MR4027588}, we obtain a generating set for $PT_5$ of cardinality 43. By removing the redundant generators, we obtain the following minimial generating set for $PT_5$: 
\begin{multicols}{3}
\begin{enumerate}
\item[] $a_{1}=(t_{1} t_{2}\big)^{3}$
\item[] $a_{2}=\big((t_{1} t_{2})^{3}\big)^{t_{3}}$
\item[] $a_{3}=\big((t_{1} t_{2})^{3}\big)^{t_{3} t_{2}}$
\item[] $a_{4}=\big((t_{1} t_{2})^{3}\big)^{t_{3} t_{2} t_{1}}$
\item[] $a_{5}=\big((t_{1} t_{2})^{3}\big)^{t_{3} t_{2} t_{1} t_{4} t_{3} t_{2}}$
\item[] $a_{6}=\big((t_{1} t_{2})^{3}\big)^{t_{3} t_{2} t_{1}t_{4} t_{3}}$
\item[] $a_{7}=\big((t_{1} t_{2})^{3}\big)^{t_{3} t_{2} t_{1} t_{4}}$
\item[] $a_{8}=\big((t_{1} t_{2})^{3}\big)^{t_{3} t_{2} t_{4} t_{3}}$
\item[] $a_{9}=\big((t_{1} t_{2})^{3}\big)^{t_{3} t_{4} t_{3} t_{2}}$
\item[] $a_{10}=\big((t_{1} t_{2})^{3}\big)^{t_{3} t_{4}}$ 
\item[] $a_{11}=(t_{2} t_{3})^{3}$
\item[] $a_{12}=\big((t_{2} t_{3})^{3}\big)^{t_{1}}$
\item[] $a_{13}=\big((t_{2} t_{3})^{3}\big)^{t_{1} t_{2}} $
\item[]  $a_{14}=\big((t_{2} t_{3})^{3}\big)^{t_{4} t_{3} t_{2} t_{1}}$
\item[] $a_{15}=\big((t_{2} t_{3})^{3}\big)^{t_{4} t_{3} t_{2}}$
\item[] $a_{16}=\big((t_{2} t_{3})^{3}\big)^{t_{4} t_{3}}$
\item[] $a_{17}=\big((t_{2} t_{3})^{3}\big)^{t_{4}}$
\item[] $a_{18}=\big((t_{2} t_{3})^{3}\big)^{t_{1} t_{2} t_{4} t_{3} t_{2} t_{1}}$
\item[] $a_{19}=\big((t_{2} t_{3})^{3}\big)^{t_{1} t_{2} t_{4} t_{3} t_{2}}$
\item[] $a_{20}=\big((t_{2} t_{3})^{3}\big)^{t_{1} t_{2} t_{4} t_{3}}$
\item[] $a_{21}=\big((t_{2} t_{3})^{3}\big)^{t_{1} t_{2} t_{4}}$
\item[] $a_{22}=\big((t_{2} t_{3})^{3}\big)^{t_{1} t_{4} t_{3} t_{2} t_{1}}$
\item[] $a_{23}=\big((t_{2} t_{3})^{3}\big)^{t_{1} t_{4} t_{3} t_{2}}$
\item[] $a_{24}=\big((t_{2} t_{3})^{3}\big)^{t_{1} t_{4} t_{3} }$
\item[] $a_{25}=\big((t_{2} t_{3})^{3}\big)^{t_{1}t_{4}}$
\item[] $a_{26}=\big((t_{3} t_{4})^{3}\big)$
\item[] $a_{27}=\big((t_{3} t_{4})^{3}\big)^{t_{2}t_{1}t_{3}t_{2}}$
\item[] $a_{28}=\big((t_{3} t_{4})^{3}\big)^{t_{2}t_{1}t_{3}}$
\item[] $a_{29}=\big((t_{3} t_{4})^{3}\big)^{t_{2}t_{1}}$
\item[] $a_{30}=\big((t_{3} t_{4})^{3}\big)^{t_{2}t_{3}}$
\item[] $a_{31}=\big((t_{3} t_{4})^{3}\big)^{t_{2}}$
\end{enumerate}
\end{multicols} 

By definition, we have $K_4 = \langle s_1s_0(c_{111}), s_2s_0(c_{111}), s_3s_0(c_{111}), s_2s_1(c_{111}), s_3s_1(c_{111}),  s_3s_2(c_{111}) \rangle.$ Direct calculation gives
\begin{eqnarray*}
s_1(x_3 x_5) = a_8 a_{16} a_{26}, && s_2(x_3 x_5) =a_{23} a_9 a_{31} a_{17},\\
s_3(x_3 x_5) = a_3 a_{19} a_{11} a_{30}, && s_2(x_6 x_2) = a_{29} a_{25} a_{10},\\
s_3(x_6 x_2) =a_{12} a_{28} a_2 a_{20}, && s_3(x_1x_7) = a_1 a_{13} a_{27}.
\end{eqnarray*} 
Thus,  $K_4$ is free of rank 6, and hence $K_4 \cong F[S^2]_4$.
\end{proof}

\begin{problem} \label{q2}
Determine a presentation of $K_n$ for $n \geq 4$.
\end{problem}

We  consider $c_{111}$ as a 2-simplex in the simplicial group $SPT_*$. Since $d_0 (c_{111}) = d_1 (c_{111} )= d_2 (c_{111}) = 1$,
there is a (unique) simplicial map
$$
\theta : S^2 \to SPT_*
$$
such that $\theta(\sigma)=c_{111}$, where $\sigma=(0,1,2)$ is the non-degenerate 2-simplex of the simplicial sphere $S^2$. By Milnor's  construction, the simplicial map $\theta$ extends uniquely to a simplicial homomorphism
$$
\Theta : F[S^2]_* \longrightarrow  SPT_*.
$$
We note that  $K_*=\Theta(F[S^2]_*)$ and it is the smallest simplicial subgroup of $SPT_*$ containing $c_{111}$. Further, by Proposition \ref{k3 and k4},
$$
\Theta_n : F[S^2]_n \longrightarrow SPT_n
$$
is injective for $n\leq 4$. If each $\Theta_n:F[S^2]_n \to SPT_n$ is injective, then by \eqref{pi1s3 and moore homotopy}, we have 
$$\pi_{n+1} (S^3) \cong Z_n(F[S^2]_*)/B_n(F[S^2]_*) \cong Z_n(K_*)/B_n(K_*) \cong \pi_n (K_*).$$  Thus, if $\Theta$ is injective, then we can describe $\pi_{n+1}(S^3)$ as a quotient of a subgroup of 
$PT_{n+1}$. For instance, the generator of $\pi_3(S^3)\cong \mathbb Z$ can be represented by the pure twin $(t_1t_2)^3$.

 It appears that the following holds.

\begin{conjecture}
$\Theta : F[S^2]_* \longrightarrow  K_*$ is an isomorphism.
\end{conjecture}
\medskip

\begin{ack}
{\rm Valeriy Bardakov is supported by the state contract of the Sobolev Institute of Mathematics, SB RAS (No. I.1.5, Project FWNF-2022-0009). Pravin Kumar is supported by the PMRF fellowship at IISER Mohali. Mahender Singh is supported by the Swarna Jayanti Fellowship grants DST/SJF/MSA-02/2018-19 and SB/SJF/2019-20/04.}
\end{ack}
\medskip

\section{declaration}
The authors declare that there is no data associated to this paper and that there are no conflicts of interests.

\bibliographystyle{plain}
\bibliography{template}

\begin{thebibliography}{1}

\bibitem{MR2966697} V. G. Bardakov, R. Mikhailov, V. Vershinin and J. Wu,  \textit{Brunnian braids on surfaces}, Algebr. Geom. Topol. 12 (2012), no. 3, 1607--1648. 

\bibitem{BW2} V. G. Bardakov and J. Wu,  \textit{Lifting theorem for the virtual pure braid groups}, Chinese Ann. Math. Ser. B (2023),  DOI: 10.1007/s11401-007-0001-x.

\bibitem{MR4027588}  V. Bardakov, M. Singh and A. Vesnin,  \textit{Structural aspects of twin and pure twin groups}, Geom. Dedicata 203 (2019), 135--154.

\bibitem{MR3482589} V. G. Bardakov, V. V. Vershinin and J. Wu, \textit{On Cohen braids}, Proc. Steklov Inst. Math. 286 (2014), no. 1, 16--32.

\bibitem{MR3876348} A. Bartholomew, R. Fenn, N. Kamada and S. Kamada,  \textit{Doodles on surfaces}, J. Knot Theory Ramifications 27 (2018), no. 12, 1850071, 26 pp.

\bibitem{MR2188127} A. J. Berrick,  F. R. Cohen,  Y. L. Wong and J. Wu, \textit{Configurations, braids, and homotopy groups}, J. Amer. Math. Soc. 19 (2006), no. 2, 265--326.

\bibitem{MR3108834} A. J. Berrick, E. Hanbury and J. Wu, \textit{Brunnian subgroups of mapping class groups and braid groups}, Proc. Lond. Math. Soc. (3) 107 (2013), no. 4, 875--906. 

\bibitem{MR3152716} A. J. Berrick, E. Hanbury and J. Wu, \textit{Delta-structures on mapping class groups and braid groups}, Trans. Amer. Math. Soc. 366 (2014), no. 4, 1879--1903.

\bibitem{MR1317619} A. Bj\"{o}rner and V. Welker, \textit{The homology of ``k-equal" manifolds and related partition lattices}, Adv. Math. 110 (1995), no. 2, 277--313. 

\bibitem{MR1349129} F. R. Cohen, \textit{On combinatorial group theory in homotopy}. Homotopy theory and its applications (Cocoyoc, 1993), 57--63,
Contemp. Math., 188, Amer. Math. Soc., Providence, RI, 1995.


\bibitem{MR2853222} F. R. Cohen and J. Wu, \textit{Artin's braid groups, free groups, and the loop space of the 2-sphere}, Q. J. Math. 62 (2011), no. 4, 891--921.

\bibitem{MR3200492} F. Duzhin and S. M. Z. Wong, \textit{On two constructions of Brunnian links}, J. Knot Theory Ramifications 23 (2014), no. 3, 1420002, 6 pp.

\bibitem{MR0547452} R. Fenn and P. Taylor, \textit{Introducing doodles}, Topology of low-dimensional manifolds (Proc. Second Sussex Conf., Chelwood Gate, 1977), pp. 37--43, Lecture Notes in Math., 722, Springer, Berlin, 1979.

\bibitem{MR2915498} G. Friedman, {\it Survey article: {A}n elementary illustrated introduction to simplicial sets}, Rocky Mountain J. Math.42(2012), no.2, 353--423.

\bibitem{MR4170471} J. Gonz\'{a}lez, J. L. Le\'{o}n-Medina and C. Roque-M\'{a}rquez,  \textit{Linear motion planning with controlled collisions and pure planar braids}, Homology Homotopy Appl. 23 (2021), no. 1, 275--296.

\bibitem{Gotin} K. Gotin, {\it Markov theorem for doodles on two-sphere}, (2018), arXiv:1807.05337.


\bibitem{MR1370644} M. Khovanov, \textit{Doodle groups}, Trans. Amer. Math. Soc. 349 (1997), 2297--2315.

\bibitem{Khovanov1990}  M. Khovanov, \textit{New geometrical constructions in low-dimensional topology}, (1990), preprint.

\bibitem{MR1386845} M.  Khovanov,  \textit{Real $K(\pi,1)$ arrangements from finite root systems}, Math. Res. Lett. 3 (1996), 261--274.

\bibitem{MR3180740} F. Lei, F. Li and J. Wu, \textit{On simplicial resolutions of framed links}, Trans. Amer. Math. Soc. 366 (2014), 3075--3093.

\bibitem{MR324684} H. Levinson, \textit{ Decomposable braids and linkages}, Trans. Amer. Math. Soc. 178 (1973), 111--126.

\bibitem{MR362287} H. Levinson,  \textit{ Decomposable braids as subgroups of braid groups}, Trans. Amer. Math. Soc. 202 (1975), 51--55.

\bibitem{MR2551462} J. Li and J. Wu, \textit{Artin braid groups and homotopy groups}, Proc. Lond. Math. Soc. (3) 99 (2009), no. 3, 521--556.

\bibitem{MR1822143} A. Maes and C. Cerf,  \textit{A family of Brunnian links based on Edwards' construction of Venn diagrams}, J. Knot Theory Ramifications 10 (2001), no. 1, 97--107.

\bibitem{MR0222892} J. P. May, \textit{Simplicial objects in algebraic topology}, Van Nostrand Mathematical Studies, No. 11 D. Van Nostrand Co., Inc., Princeton, N.J.-Toronto, Ont.-London 1967 vi+161 pp.

\bibitem {Milnor} J. Milnor, \textit{On the construction $F[K]$}, Algebraic Topology - A Student Guide, by J. F. Adams, Cambridge Univ. Press, 1972, 119--136.

\bibitem{MR3035327} R. V. Mikhailov and J. Wu, \textit{Combinatorial group theory and the homotopy groups of finite complexes}, Geom. Topol. 17 (2013), 235--272.

\bibitem{moore} J. C. Moore, {\it Homotopie des complexes mon\"{o}id\'{e}aux}, S\'{e}minaire Henri Cartan (1954-55).

\bibitem{MR4145210} T. K. Naik, N. Nanda and M. Singh, \textit{Conjugacy classes and automorphisms of twin groups}, Forum Math. 32 (2020), no. 5, 1095--1108.

\bibitem{MR4651964} T. K. Naik, N. Nanda and M. Singh,  \textit{Structure and automorphisms of pure virtual twin groups}, Monatsh. Math.  202 (2023), 555--582.

\bibitem{NNS2} T. K. Naik, N. Nanda and M. Singh, \textit{Virtual planar braid groups and permutations}, J. Group Theory (2023), 30 pp, https://doi.org/10.1515/jgth-2023-0010.

\bibitem{MR1955357} J. Wu,  \textit{Homotopy theory of the suspensions of the projective plane}, Mem. Amer. Math. Soc. 162 (2003), no. 769, x+130 pp.

\bibitem{MR2203532} J. Wu,  \textit{On maps from loop suspensions to loop spaces and the shuffle relations on the Cohen groups}, Mem. Amer. Math. Soc. 180 (2006), no. 851, vi+64 pp.

\end{thebibliography}

\end{document}